\documentclass{article}
\title{Finitary codings and stochastic domination for Poisson representable processes}
\author{Yinon Spinka\thanks{Tel Aviv University. Research supported in part by ISF grant 1361/22}}
\date{June 2025}

\usepackage{amsmath,amsfonts,amssymb,amsthm,graphics,enumerate,mathtools}
\usepackage{hyperref}
\hypersetup{
    colorlinks=true,
    linkcolor=blue,
    citecolor=red,
    urlcolor=blue,
    pdfborder={0 0 0}
}

\usepackage{graphicx,xcolor,bbm,fullpage}

\usepackage{cleveref}
  \crefname{theorem}{Theorem}{Theorems}
  \crefname{thm}{Theorem}{Theorems}
  \crefname{lemma}{Lemma}{Lemmas}
  \crefname{lem}{Lemma}{Lemmas}
  \crefname{remark}{Remark}{Remarks}
  \crefname{prop}{Proposition}{Propositions}
  \crefname{proposition}{Proposition}{Propositions}
  \crefname{notation}{Notation}{Notations}
  \crefname{claim}{Claim}{Claims}
  \crefname{observation}{Observation}{Observations}
  \crefname{defn}{Definition}{Definitions}
  \crefname{corollary}{Corollary}{Corollaries}
  \crefname{section}{Section}{Sections}
  \crefname{figure}{Figure}{Figures}
  \crefname{exercise}{Exercise}{Exercises}
  \crefname{example}{Example}{Examples}
  \crefname{assumption}{Assumption}{Assumptions}

\newtheorem{thm}{Theorem}[section]

\newtheorem{claim}[thm]{Claim}
\newtheorem{lemma}[thm]{Lemma}
\newtheorem{corollary}[thm]{Corollary}

\newtheorem{example}[thm]{Example}

\numberwithin{equation}{section}

\theoremstyle{definition}
\newtheorem{remark}[thm]{Remark}

\def\cZ{\mathcal{Z}}

\def\cX{\mathcal{X}}

\def\cP{\mathcal{P}}

\def\cM{\mathcal{M}}

\def\cJ{\mathcal{J}}
\def\cI{\mathcal{I}}

\def\cF{\mathcal{F}}

\def\cC{\mathcal{C}}
\def\cB{\mathcal{B}}
\def\cA{\mathcal{A}}
\def\P{\mathbb{P}}
\def\E{\mathbb{E}}

\def\Z{\mathbb{Z}}

\newcommand{\1}{\mathbf{1}}
\def\eps{\varepsilon}

\DeclareMathOperator{\dist}{dist}
\DeclareMathOperator{\diam}{diam}

\begin{document}
\maketitle

\begin{abstract}
Construct a random set by independently selecting each finite subset of the integers with some probability depending on the set up to translations and taking the union of the selected sets.
We show that when the only sets selected with positive probability are pairs, such a random set is a finitary factor of an IID process, answering a question of Forsstr{\"o}m, Gantert and Steif~\cite{forsstrom2024poisson}. More generally, we show that this is the case whenever the distribution induced by the size of the selected sets has sufficient exponential moments, and that the existence of some exponential moment is necessary. We further show that such a random set is stochastically dominated by a non-trivial Bernoulli percolation if and only if there is a finite exponential moment, thereby partially answering another question from~\cite{forsstrom2024poisson}. We also give a partial answer to a third question from~\cite{forsstrom2024poisson} regarding a form of phase transition. These results also hold on $\Z^d$ with $d \ge 2$.
In the one-dimensional case, under the condition that the distribution induced by the diameter of the selected sets has an exponential moment, we further show that such a random set is finitarily isomorphic to an IID process.
\end{abstract}

\section{Introduction}

Let $X=(X_A)_{A \Subset V}$ be independent Bernoulli random variables, where $A$ ranges over the finite (non-empty) subsets of a countable set $V$.
Define
\[ \bar X := \bigcup \{ A \Subset V : X_A = 1 \} .\]
Then $\bar X$ is a random subset of $V$, obtained by independently including each finite set $A$ with some probability, and then taking the union of all the included sets. We regard $\bar X$ as a random element of $\{0,1\}^V$, or equivalently, as a $\{0,1\}$-valued process $(\bar X_v)_{v \in V}$ on $V$. Such a process is an example of what has been called a Poisson representable process, recently introduced by Forsstr{\"o}m, Gantert and Steif in~\cite{forsstrom2024poisson} (our setting corresponds to the generic situation where the intensity measure concentrates on finite sets).
In the present paper, we investigate its stochastic domination and finitary coding properties.
Throughout the paper we denote $p_A := \P(X_A=1)$, and assume that $p_A<1$ for all $A \Subset V$ and that $\sum_{A \Subset V: v\in A} p_A < \infty$ for all $v \in V$. These natural assumptions are equivalent to $\P(\bar X_v=1)<1$ for all $v \in V$.

Suppose now that the underlying space is $V = \Z^d$ for some $d \ge 1$, and that the probabilities $(p_A)$ are invariant to translations, i.e., $p_{A+v}=p_A$ for all $A \Subset \Z^d$ and $v \in \Z^d$. Clearly, the process $\bar X$ is now invariant to translations. It is also easy to see that $\bar X$ is a factor of an IID process. The main question of interest in this paper, raised by Forsstr{\"o}m, Gantert and Steif in the special case when $p_A=0$ unless $|A|=2$, is whether $\bar X$ is a \emph{finitary} factor of an IID process (see~\cite[Question 7]{forsstrom2024poisson}). Our main result gives a positive answer to this.

\begin{thm}\label{thm:ffiid-under-all-exp-moment}
    Let $V=\Z^d$, $d \ge 1$. Suppose that the probabilities $(p_A)$ are translation invariant and that
 \begin{equation}\label{eq:exp-moment-all}
  \sum_{A \Subset \Z^d : 0 \in A} p_A e^{\lambda |A|} < \infty \qquad\text{for all }\lambda>0 .
 \end{equation}
    Then $\bar X$ is a finitary factor of an IID process.
\end{thm}

In the converse direction, we prove the following.
\begin{thm}\label{thm:ffiid-implies-some-exp-moment}
    Let $V=\Z^d$, $d \ge 1$. Suppose that the probabilities $(p_A)$ are translation invariant and that $\bar X$ is a finitary factor of an IID process.
    Then
 \begin{equation}\label{eq:exp-moment}
  \sum_{A \Subset \Z^d : 0 \in A} p_A e^{\lambda |A|} < \infty \qquad\text{for some }\lambda>0 .
 \end{equation}
\end{thm}

We thus have a gap between the sufficient condition given by \cref{thm:ffiid-under-all-exp-moment} and the necessary condition given by \cref{thm:ffiid-implies-some-exp-moment}. It turns out that neither condition is both necessary and sufficient.
That a single exponential moment is not sufficient is demonstrated in \cref{ex:single-exp-moment-not-sufficient} below.
On the other hand, the next result shows that not all exponential moments are needed when $p_{\{0\}}>0$, and it suffices that the exponential moment corresponding to $\lambda = -\log p_{\{0\}}$ exists.

\begin{thm}\label{thm:ffiid-with-large-p1}
    Let $V=\Z^d$, $d \ge 1$. Suppose that the probabilities $(p_A)$ are translation invariant and that
 \begin{equation}\label{eq:exp-moment-p1}
  p_{\{0\}}>0 \qquad\text{and}\qquad \sum_{A \Subset \Z^d : 0 \in A} p_A \cdot p_{\{0\}}^{-|A|} < \infty .
 \end{equation}
    Then $\bar X$ is a finitary factor of an IID process.
\end{thm}
Let us mention that this result implies that for some choices of $(p_A)$ there is a non-trivial phase transition as one varies the probability of singletons (i.e., varying the common value $p_{\{v\}}$, $v \in V$, while keeping all other $p_A$ fixed), whereby $\bar X$ is a finitary factor of IID for values above some threshold, but not below.
This shows that being a finitary factor of IID is not determined merely by the ``tail behavior'' of the $p_A$, but rather by some interplay between the values of $p_A$ for ``small sets'' (or at least the singletons) and ``large sets''. This has the consequence of a certain phase transition (as the underlying intensity measure is scaled) for the property of being a finitary factor of IID, shedding light on another question from~\cite{forsstrom2024poisson}.
Let us also mention that while the assumptions of \cref{thm:ffiid-under-all-exp-moment,thm:ffiid-with-large-p1} are somewhat similar, the assumption of the first theorem actually implies a stronger result (in which a richer process is a finitary factor of IID), which, in general, is not true under the assumption of the latter theorem.
See the discussion for further details on these issues.


We now present a result providing a finitary coding with exponential tail for the coding radius.
This requires an assumption beyond the existence of an exponential moment for the size of $A$.
One possibility is to only allow connected sets $A$.
A different possibility is to control the tail of the diameter of $A$.
We formulate our condition in a way that captures both possibilities, by taking into account the ``connected size'' of $A$.
For a set $A$, let $|{\sf c}(A)|$ be the minimal size of a connected set containing $A$.

\begin{thm}\label{thm:ffiid-connected}
    Suppose that $V$ is the vertex set of an infinite connected transitive graph $G$ of degree $\Delta \ge 2$. Let $\Gamma$ be a closed transitive group of automorphisms of $G$.
    Suppose that the probabilities $(p_A)$ are $\Gamma$-invariant and
    \begin{equation}\label{eq:exp-moment-cA}
     \sum_{A \Subset V : v \in A} p_A e^{\lambda |{\sf c}(A)|} < \infty \qquad\text{for some }\lambda>\lambda_* := \log(3\Delta-1) \text{ and } v \in V .
    \end{equation}
    Then $\bar X$ is a finitary $\Gamma$-factor of an IID process with a coding radius having exponential tail. Moreover, in the case when $V=G=\Gamma=\Z$, the result holds with $\lambda_*=0$.
\end{thm}

Note that in the case when $V=G=\Gamma=\Z$, the required condition is that the diameter of $A$ has some exponential moment.
We prove further results on $\Z$, showing that $\bar X$ is a hidden (countable-state) Markov chain, and that this richer process is finitarily isomorphic to an IID process. See \cref{sec:ffiid-one-dim}.

We now turn to a different aspect of the process $\bar X$, regarding stochastic domination.
For translation invariant Poisson representable processes on $V=\Z^d$, Forsstr{\"o}m, Gantert and Steif gave a precise criteria for when there is stochastic domination from below by an IID process of density $p$. In our setting, this yields that $\bar X$ always stochastically dominates a non-trivial IID process. In that paper, the authors ask what can be said concerning stochastic domination from above by IID processes. This is addressed in the following result.

\begin{thm}\label{thm:stochastic-domination}
 Suppose that the probabilities $(p_A)$ are invariant under a transitive group of permutations $\Gamma$ on $V$.
 Then $\bar X$ is stochastically dominated by a non-trivial IID process whenever
 \begin{equation}\label{eq:dom-exp-moment}
  \sum_{A \Subset V : v \in A} p_A e^{\lambda |A|} < \infty \qquad\text{for some }\lambda>0 \text{ and } v \in V .
 \end{equation}
 In the case when $V=\Gamma=\Z^d$, $d \ge 1$, this condition is also necessary.
\end{thm}

\subsection{Discussion and proof overview}

As we have mentioned, the processes $\bar X$ considered here are special cases of what is called Poisson representable processes. Let us describe what these are more generally. Let $\nu$ be a $\sigma$-finite measure on $\cP(V) \setminus \{\emptyset\}$, the space of non-empty subsets of~$V$. Let $Y^\nu$ be the Poisson process on $\cP(V) \setminus \{\emptyset\}$ with intensity measure~$\nu$. Then $Y^\nu$ is a random (multi-)collection of non-empty subsets of $V$. Now let $X^\nu$ denote the union of these sets. This yields a random subset $X^\nu$ of $V$, and as before, we may regard it as a $\{0,1\}$-valued process on~$V$. A process that can be obtained in this way is called a \textbf{Poisson representable process}. This class of processes was recently introduced by Forsstr{\"o}m, Gantert and Steif in~\cite{forsstrom2024poisson}, where the main focus was to determine which processes belong to this class. Some mixing properties of such processes have been investigated in~\cite{bethuelsen2025mixing}.

When $\nu$ concentrates on finite subsets of $V$, i.e., when $\nu(\{ A \subset V : |A|=\infty \})=0$, then $X^\nu$ can be constructed more concretely in the manner described in the introduction. Namely, setting $p_A := 1 - e^{-\nu(\{A\})}$ for $A \Subset \Z^d$, we have that $Y^\nu$ (regarded here as a collection with no repetitions) has the same distribution as $\{ A \Subset \Z^d : X_A = 1\}$, and hence, $X^\nu$ has the same distribution as $\bar X$. In this sense, the processes $\bar X$ considered in this paper are generic Poisson representable processes whose intensities concentrate on finite subsets.

Our main focus in this paper is on the stationary case, where $V=\Z^d$ for some $d \ge 1$ and the family of probabilities $(p_A)$ (or equivalently the intensity measure $\nu$) is invariant under translations, in which case the process $\bar X$ is also translation invariant. Recall that a process $Z=(Z_v)_{v \in \Z^d}$ is a \textbf{factor of IID} if it can be expressed as an equivariant function of IID random variables, i.e., if $Z$ has the same distribution as $F(U)$, where $U=(U_v)_{v \in \Z^d}$ are IID, and $F$ is some measurable function such that $F(U+v)=F(U)+v$ almost surely for all $v \in \Z^d$. The process is a \textbf{finitary factor of IID (FFIID)} if the function $F$ can further be chosen so that $F(U)_0$ is almost surely determined by a random, but finite, portion of $U$. More precisely, if there exists a stopping time $R$ with respect to the filtration $\cF_n := (U_v)_{\|v\| \le n}$ (the choice of norm here is not important for our purposes), such that $F(U)_0$ is measurable with respect to $\cF_\tau$. Such an $R$ is called a coding radius.

In the stationary case, it is easy to see that $\bar X$ is always a factor of IID.
This is because we may regard each random variable $X_A$ as ``attached'' to a specific point in $\Z^d$, say to $\min A$, the minimal element of $A$ in the lexicographic order.
More specifically, let $(U_v)_{v \in \Z^d}$ be IID, with each $U_v$ consisting of a sequence of independent Bernoulli random variables $U_{v,A}$, $A \Subset \Z^d$, $\min A = v$, of parameter $p_A$. Define $F(U)_w := \max \{ U_{v,A} : u \in v+A \}$. It is straightforward to check that $F$ is equivariant and that $F(U)=\bar X$ in distribution.

The main question of concern in this paper is whether or not $\bar X$ is a \emph{finitary} factor of IID. Note that if we only allow to include sets of bounded diameter, i.e., $\sup \{ \diam A : p_A>0 \} < \infty$, then the map described above has bounded range, so that $\bar X$ is a block factor of IID. When the diameter is unbounded, it is not hard to see that $\bar X$ is not finitely dependent, and hence not a block factor of IID.
Suppose now that the allowed sets have unbounded diameter, but bounded size, i.e., $\sup \{ |A| : p_A>0 \} < \infty$. In this case, it is evident that the ``natural'' map described above is not finitary. However, this does not mean that the process $\bar X$ itself is not FFIID, as there may be an alternative construction via a map which is finitary. The problem we are concerned with is determining whether there is such an alternative construction.

The problem is already interesting in one dimension and in the situation in which we only allow to include sets of size two, i.e., $p_A=0$ unless $|A|=2$. In fact, this is the setting of the problem posed by Forsstr{\"o}m, Gantert and Steif in \cite[Question 7]{forsstrom2024poisson}. In this case, the parameters of the process are described by the sequence of numbers $(p_n)_{n=1}^\infty$ given by $p_n:=p_{\{0,n\}}$. Recall that we assume that $p_n<1$ for all $n$ and that $\sum_n p_n < \infty$, and that this is necessary and sufficient for $\P(\bar X \equiv 1) = 0$.

Let us first explain a simple construction of $\bar X$ as a finitary factor of IID under the assumption that $\sum_n \sqrt{p_n} < \infty$. In this case, we may ``split'' the randomness of $X_{\{i,i+n\}}$ into a pair of independent coin tosses of success probability $\sqrt{p_n}$ each, one of which we regard as ``attached'' to the endpoint $i$, and the other to $i+n$. Since $\sum_n \sqrt{p_n} < \infty$, the number of successful coin tosses attached to a given site $i$ is almost surely finite, and in order to determine the value of $\bar X_i$, one only needs to inspect the corresponding paired coin tosses. Since we only inspect finitely many pairs (and the pairs which we need to inspect are determined locally), this describes $\bar X$ as a finitary factor of the IID coin tosses. More precisely, let $(U_i)_{i \in \Z}$ be IID, with each $U_i$ consisting of a sequence of independent Bernoulli random variables $U^\pm_{i,n}$, $n \ge 1$, of parameter $\sqrt{p_n}$. Define $F(U)_i := \max \{ \max \{ U^+_{i,n} U^-_{i+n,n},U^-_{i,n} U^+_{i-n,n}\} : n \ge 1 \}$. It is straightforward to check that $F$ is equivariant and that $F(U)=\bar X$ in distribution (think of $X_{\{i,i+n\}}$ as $U^+_{i,n} U^-_{i+n,n}$).

When $\sum_n \sqrt{p_n} = \infty$, the same construction is still possible and yields $\bar X$ as a factor of IID, but the map is not finitary (there will be infinitely many successful coin tosses attached to any given site $i$). The idea here is to modify the construction to consider many pairs $\{i,i+n\}$ simultaneously, and couple them together in a clever way using stochastic domination tools we develop.
Let us first give a high-level explanation of the approach (which is also relevant for the general case).

Suppose that $\cA_1,\cA_2,\dots$ is some decomposition of the allowed sets $A \Subset \Z^d$ into translation invariant families (by the latter we just mean that $A \in \cA_k$ implies that $A+n \in \cA_k$). In the language of Poisson representable processes, this corresponds to a decomposition $\nu=\nu_1+\nu_2+\cdots$ of the intensity measure $\nu$ into invariant measures. Let $\bar X^k$ denote the process corresponding to $\cA_k$, i.e., $\bar X^k := \bigcup \{ A \in \cA_k : X_A=1 \}$ (this is $X^{\nu_k}$ in the language of Poisson representable processes).
Clearly, $\bar X = \max_k \bar X^k$, with the maximum taken coordinate-wise. Suppose now that each $\cA_k$ consists of bounded subsets (i.e., of sets of diameter at most some $d_k$). Then, as we explained earlier, it is easy to to see that each $\bar X^k$ is a block factor of IID. This immediately implies that the maximum of all $\bar X^k$ is a factor of IID, but it tells us nothing about the finitariness of the resulting factor.

The idea will be to show that there is a decomposition as above, such that each $\bar X^k$ is not merely a block factor of IID, but rather is a block factor of IID with the additional property that the coding radius is zero with probability at least $1-\epsilon_k$, for any pre-chosen sequence $(\epsilon_k)$ of positive numbers. Choosing a summable sequence, we shall conclude via the Borel--Cantelli lemma that almost surely only finitely many of the coding radii are positive. From this it follows that the maximum of the $\bar X^k$ is a finitary factor of IID. Hence, $\bar X$ is FFIID.

Toward establishing the above, loosely speaking, we split the family $\cA_k$ into countably many finite collections, and show for each collection $\cF$ that $\bar X^\cF:=(\max\{ X_A : A \in \cF,\, i \in A\})_i$ is stochastically dominated by independent Bernoulli random variables (these are analogous to the coin tosses described earlier). In this manner, each $v$ has many independent coin tosses attached to it (one coming from each collection $\cF$; of course, if $i$ does not belong to any set in $\cF$ then the associated Bernoulli random variable has parameter zero), and the aim is to make the sum of their parameters less than $\eps_k$ (this ensures that the coding radius is positive with probability at most $\eps_k$). Since these coin tosses dominate the actual value of $\bar X^\cF_i$, only for the coin tosses which are successful do we need to further inspect the underlying monotone coupling between $\bar X^\cF$ and the coin tosses. This ensures that the coding radius is almost surely finite, i.e., that the construction is finitary.

In the case where we only allow to include sets of size two, each collection $\cF$ as above will consist of pairs $\{i,i+n\}$ having a common value of $i$, and this approach will give a rather clean description of $\bar X$ as a finitary factor of IID.
In the case where we allow to include sets of size three, or sets of some bounded size, this approach will still work, but the collections $\cF$ will not be so simple and this will lead to the outcome that $\bar X^k$ is only a finitary factor of IID (rather than a block factor as before) with the additional aforementioned property that the coding radius is typically zero. Note that this has no bearing on the final conclusion that $\bar X$ is FFIID.
Extending this to the case where sets of unbounded size are allowed, but~\eqref{eq:exp-moment-all} holds, is just a matter of choosing the decomposition appropriately and does cause further trouble. The extension to $\Z^d$ is also not much of a complication. This will lead to a proof of \cref{thm:ffiid-under-all-exp-moment}.

It is also worth noting that these ideas yield more than just that $\bar X$ is FFIID. The proof in fact shows that under the assumptions of \cref{thm:ffiid-under-all-exp-moment}, a much richer process, denoted $\hat X$, is FFIID.
This is the process which, loosely speaking, records the included sets that contain a given site. Precisely, let $\cA_0$ be the collection of finite subsets of $\Z^d$ that contain 0, let $\cP_{\text{fin}}(\cA_0)$ be the collection of finite subsets of $\cA_0$, and define a $\cP_{\text{fin}}(\cA_0)$-valued process $\hat X$ by
\begin{equation}\label{eq:hat-X}
 \hat X_v := \{ A \in \cA_0 : X_{A+v}=1 \} .
\end{equation}
Observe that $\bar X$ is a simple block factor of $\hat X$, namely, $\bar X_v = \1\{\hat X_v \neq \emptyset\}$. In particular, $\bar X$ is FFIID whenever $\hat X$ is such.

In the converse direction, we have \cref{thm:ffiid-implies-some-exp-moment} which states that an exponential moment for the size of the included sets is necessary in order for $\bar X$ to be FFIID. This relies on our results regarding stochastic domination of $\bar X$ by IID processes, which imply that when~\eqref{eq:exp-moment} does not hold, the probability that $\bar X$ is identically 1 on a large box decays slower than exponentially in the volume of the box. This is at odds with the fact that a finitary factor of IID satisfies the mean ergodic theorem with an exponential rate.
This leads to a proof of \cref{thm:ffiid-implies-some-exp-moment}.

There is a gap between the sufficient condition given by \cref{thm:ffiid-under-all-exp-moment} and the necessary condition given by \cref{thm:ffiid-implies-some-exp-moment}.
While we have not been able to close this gap completely, we have some results in this direction. We begin with the following example that shows that~\eqref{eq:exp-moment} is not a sufficient condition.

\begin{example}[a single exponential moment is not sufficient]\label{ex:single-exp-moment-not-sufficient}
We give an example demonstrating that~\eqref{eq:exp-moment} is not a sufficient condition for $\bar X$ to be FFIID.
Fix $\lambda>0$ and let $k_0$ be large enough so that $\sum_{k=k_0}^\infty k e^{-\lambda k} < e^{-\lambda}$. Let $p_A:=e^{-\lambda |A|}$ if $A$ is a translation of $A_k := \{0,k,k,\dots,(k-1)k\}$ for some $k \ge k_0$, and let $p_A:=0$ otherwise. Then clearly~\eqref{eq:exp-moment} holds (with any smaller $\lambda$), but~\eqref{eq:exp-moment-all} does not. Note also that $\P(\bar X_0=1) < e^{-\lambda}$. Suppose that $\bar X$ is a factor of IID and let $R$ denote the coding radius. Then
\[ e^{-\lambda k} = p_{A_k} \le \P(\bar X_{A_k} \equiv 1) \le (\P(\bar X_0=1) + \P(R \ge k/2))^k .\]
Thus, $\P(\bar X_0=1) + \P(R \ge k/2) \ge e^{-\lambda}$, and taking $k$ to infinity, we see that $\P(R=\infty) \ge e^{-\lambda} - \P(\bar X_0=1) > 0$, so that the factor is not finitary.
Note that this conclusion holds also if we increase some of the $p_A$, as long as we preserve the property that $\P(\bar X_0=1) < e^{-\lambda}$. In particular, this also provides an example with $p_{\{0\}}>0$.
\end{example}

The example shows that even having a strong exponential moment, in the sense that~\eqref{eq:exp-moment} holds for a large $\lambda$, is not sufficient.
Note, however, that the example hinges on the fact that $\P(\bar X_0=1)$ is relatively small, namely, less than $e^{-\lambda}$. Could it be that something of this sort is true in any counterexample? Of course, $\P(\bar X_0=1) \ge p_1 := p_{\{0\}}$, so that, for this example to work, $p_1$ must be small. \cref{thm:ffiid-with-large-p1} implies that this latter property is required in any counterexample. Indeed, the theorem implies that if we fix $(p_A)_{|A|>1}$ satisfying~\eqref{eq:exp-moment} (which does not depend on $p_{\{0\}}$), and begin varying the common value $p_1$ of $p_{\{v\}}$, $v \in \Z^d$, then for $p_1$ sufficiently close to 1, we always have that $\bar X$ is FFIID.

Note that if $\bar X$ is FFIID for some value of $p_1$, then it is also so for any larger value of $p_1$. This is because increasing $p_1$ has the effect of independent sprinkling, that is, $\bar X$ for $p_1=p+\eps$ is obtained from $\bar X$ for $p_1=p$ by taking its union with an independent Bernoulli IID process of parameter $\eps/(1-p)$. This yields the existence of a critical value $p_c \in [0,1]$ (which depends on $(p_A)_{|A|>1}$) such that $\bar X$ is FFIID when $p_1>p_c$, but not when $p_1<p_c$. \cref{thm:ffiid-with-large-p1} implies that we always have $p_c<1$ when~\eqref{eq:exp-moment} holds. Of course, by \cref{thm:ffiid-under-all-exp-moment}, $p_c=0$ when~\eqref{eq:exp-moment-all} holds. When~\eqref{eq:exp-moment} holds but~\eqref{eq:exp-moment-all} does not, both $p_c=0$ and $p_c>0$ are possible. The possibility of the former is a consequence of \cref{thm:ffiid-connected}, and the possibility of the latter was demonstrated by~\cref{ex:single-exp-moment-not-sufficient}.

This last discussion is related to another question raised in~\cite{forsstrom2024poisson} (see Question 6 there). The question asks, for a Poisson representable process $X^\nu$ with intensity measure $\nu$, whether the behavior of $X^{c\nu}$ can depend in an essential way on $c$.
A slight modification of the previous discussion shows that its behavior with regards to finitary factors can indeed depend on $c$. Recall that the correspondence between $\nu$ and $(p_A)$ is that $p_A=1-e^{-\nu(\{A\})}$. Thus, $\nu \mapsto c\nu$ translates to $p_A \mapsto p'_A:=1-(1-p_A)^c$. Let $(p_A)$ be as in \cref{ex:single-exp-moment-not-sufficient}, and recall that we may take $p_1>0$ in that example. Let $\nu$ be the corresponding measure. Then $X^\nu$ is not FFIID, but $X^{c\nu}$ is for all $c$ large enough. To see the latter, choose $\lambda>0$ for which~\eqref{eq:exp-moment} holds, and note that $p'_A \le c \cdot p_A$ for $c>0$, so that~\eqref{eq:exp-moment} continues to hold for the perturbed $(p'_A)$ with the same $\lambda$. Thus, \cref{thm:ffiid-with-large-p1} implies that $X^{c\nu}$ is FFIID whenever $1-(1-p_1)^c \ge e^{-\lambda}$.

Let us mention an important difference between \cref{thm:ffiid-under-all-exp-moment} and \cref{thm:ffiid-with-large-p1}.
As formulated, \cref{thm:ffiid-with-large-p1} implies \cref{thm:ffiid-under-all-exp-moment} in the case when $p_1>0$.
However, besides allowing for $p_1=0$, recall from the earlier discussion that in the latter theorem the conclusion can be strengthened to hold for the richer process $\hat X$ described above. Crucially, this is not true of \cref{thm:ffiid-with-large-p1}. Indeed, when $p_c>0$, it cannot be that $\hat X$ is FFIID (regardless of the actual value of $p_1$). This also means that the proof of \cref{thm:ffiid-with-large-p1} must follow a different approach than that described earlier, one which utilizes the fact that when $X_{\{i\}}=1$ we do not need to know the values of $(X_A)_{|A|>1, i \in A}$ in order to know $\bar X_i$, which must of course equal 1.

Let us now turn to \cref{thm:ffiid-connected}, which states that under a certain exponential moment condition on the ``connected size'' of the included sets, $\bar X$ is FFIID, and furthermore, the coding radius has exponential tail. For the sake of this discussion, let us assume that we only allow to include sets $A$ that are connected. This yields a Markovian property for $\bar X$. Namely, it satisfies the domain Markov property in the case when the boundary condition is all zero. This, together with additional stochastic domination tools, allows us to apply a general result from~\cite{ray2023characterizations} about finitary codings for such processes to eventually deduce the desired result for $\bar X$. The requirement in~\eqref{eq:exp-moment-cA} that $\lambda>\lambda_*$, as well as the value of $\lambda_*$, comes from the assumptions needed in order to apply the result from~\cite{ray2023characterizations}. We do not know whether this can be relaxed. While this applies in any dimension, more can be said in the one dimensional case. Here, the process $\bar X$ is in fact a hidden Markov chain. More specifically, we shall show that it is a block factor of a countable-state mixing Markov chain having exponential return times. The underlying Markov chain is simple to describe: it is the process which, loosely speaking, records the included sets that ``cross over'' a given site (compare this with the earlier $\hat X$ which records the included sets that \emph{contain} a given site). A result of Angel and the author~\cite{angel2021markov} states that any such Markov chain (countable-state mixing Markov chain having exponential return times) is FFIID with a coding length having exponential tail. Using this we conclude that the theorem holds with $\lambda_*=0$ in the one-dimensional case. Finally, a result of Rudolph~\cite{rudolph1982mixing} states that any such Markov chain is finitarily isomorphic to an IID process. We gave a new proof of this fact~\cite{spinka2025newprooffinitaryisomorphism}, which is more direct and does require the complicated machinery developed by Rudolph in~\cite{rudolph1981characterization}. Let us stress that the finitary isomorphism result here is for the Markov chain, and not for $\bar X$ which is only a block factor of the Markov chain.

Let us now turn to the stochastic domination aspect.
It turns out that our constructions of finitary codings are strongly aided by suitable stochastic domination properties (as hinted at before).
The stochastic domination tools we develop for those purposes also yield independent results which are interesting in their own right.
For stochastic domination of $\bar X$ from below by an IID Bernoulli process, an exact condition was obtained in~\cite{forsstrom2024poisson} (the result there is for general Poisson representable processes). In our setting, it says that $\bar X$ stochastically dominates an IID Bernoulli process of parameter $p$ if and only if $\P(\bar X_B \equiv 0) \le (1-p)^{|B|}$ for every box $B \subset \Z^d$. It is easy to see that this must hold for some $p>0$, so that $\bar X$ always stochastically dominates a non-trivial IID process. Forsstr{\"o}m, Gantert and Steif asked~\cite[Question 8]{forsstrom2024poisson} what one can say about stochastic domination from above by IID processes. \cref{thm:stochastic-domination} gives an answer to this for the processes consider here (recall that these are generic Poisson representable processes whose intensities concentrate on finite sets). We do not have a precise characterization of which IID processes stochastically dominate $\bar X$, but our proof gives upper and lower bounds on the parameter $p$ for which this is possible.

\bigskip
\noindent\textbf{Open problems.}

\begin{enumerate}
    \item Find a necessary and sufficient condition for $\bar X$ to be a finitary factor of IID.
    \item Find a necessary and sufficient condition for $\hat X$ (defined in~\eqref{eq:hat-X}) to be a finitary factor of IID.
\end{enumerate}

\medskip
\noindent\textbf{Organization.}
In \cref{sec:stoc-dom} we prove \cref{thm:stochastic-domination} about stochastic domination and develop the stochastic domination tools needed for our finitary coding results.
In \cref{sec:finitary-coding} we prove our results about finitary codings, starting with \cref{thm:ffiid-implies-some-exp-moment} about the necessity of an exponential moment for the existence of a finitary coding (\cref{sec:no-finitary-coding}), next the special case of \cref{thm:ffiid-under-all-exp-moment} in which only pairs are allowed to be included (\cref{sec:pairs}), and then the general case of \cref{thm:ffiid-under-all-exp-moment} about the existence of a finitary coding when all exponential moments exist (\cref{sec:finitary-coding-all-exp-moments}).
We then give the proof of \cref{thm:ffiid-with-large-p1} in \cref{sec:ffiid-p1}, and the proof of \cref{thm:ffiid-connected} in \cref{sec:connected}. Finally, the one-dimensional case is considered in more depth in \cref{sec:ffiid-one-dim}, where we prove a finitary isomorphism result for a related process.

\bigskip
\noindent\textbf{Acknowledgments.}
We thank Lev Buhovsky and Tom Meyerovitch for useful discussions.

\section{Stochastic domination}\label{sec:stoc-dom}

We continue to use the notation from before. Recall that $V$ is a countable set, $(X_A)_{A \Subset V}$ are independent Bernoulli random variables, $(p_A)$ are their parameters, and $\bar X_v=\max_{A \ni v} X_A$.

The proof of \cref{thm:stochastic-domination} relies on the following lemma.

\begin{lemma}\label{lem:domination2}
     Let $\eps_v \in (0,1]$, $v \in V$, be arbitrary, and define $\delta_v := \epsilon_v + (1-\eps_v)\sum_{A \ni v} p_A \cdot \prod_{u \in A \setminus \{v\}} \eps_u^{-1}$.
     Let $Z=(Z_v)_{v \in V}$ be independent Bernoulli random variables with parameters $(\delta_v)$.
     Then $\bar X$ is stochastically dominated by $Z$.
\end{lemma}

Note that when $\eps_v=\eps$ for all $v$, we have that $\delta_v=\eps+\eps(1-\eps)\sum_{A \ni v} p_A \cdot \eps^{-|A|}$, which relates with~\eqref{eq:dom-exp-moment}.

A common method for showing stochastic domination of a $\{0,1\}$-valued process $X=(X_v)$ by independent Bernoulli random variables $Z=(Z_v)$ is to try and show that $\P(X_v=1 \mid (X_u)_{u \neq v}) \ge \P(Z_v=1)$ almost surely for all $v$. It is standard that this implies the desired stochastic domination. This approach does not always work (indeed, in our case it does not work).
Instead, one may try to introduce an intermediate process $Y$ that is sandwiched in between $X$ and $Z$. Specifically, we shall construct a process $Y$ by an independent sprinkling on top of $X$. This obviously yields that $X \le_{st} Y$, and we shall then show that $Y \le_{st} Z$ via the previous standard approach. We extend the conclusion of \cref{lem:domination2} to include this stronger statement, which will itself be useful for us later on.

\begin{lemma}\label{lem:domination3}
    Let $\eps_v \in (0,1]$, $v \in V$, be arbitrary, and define $\delta_v := \epsilon_v + (1-\eps_v)\sum_{A \ni v} p_A \cdot \prod_{u \in A \setminus \{v\}} \eps_u^{-1}$.
    Let $(\xi_v)_{v \in V}$ be independent Bernoulli random variables of parameters $(\eps_v)$, independent of $(X_A)$, and define a process $Y=(Y_v)_{v \in V}$ by $Y_v := \max\{\bar X_v,\xi_v\}$.
    Let $Z=(Z_v)_{v \in V}$ be independent Bernoulli random variables with parameters $(\delta_v)$.
    Then $Y$ is stochastically dominated by $Z$.
    Moreover, for any $v \in V$, almost surely,
    \[ \P(Y_v=1 \mid (X_A)_{A \not\ni v},~ (Y_u)_{u \neq v}) \le \delta_v .\]
\end{lemma}
\begin{proof}
    Since $\xi_v$ is independent of $X$ and $(Y_u)_{u \neq v}$, it suffices to show that
    \[ \P(\bar X_v=1 \mid (X_A)_{A \not\ni v},~ (Y_u)_{u \neq v}) \le \sum_{A \ni v} p_A \cdot \prod_{u \in A \setminus \{v\}} \eps_u^{-1} .\]
    Since $\bar X_v \le \sum_{A \ni v} X_A$, it in turn suffices to show that, for any $A$ containing $v$,
    \[ \P(X_A=1 \mid (X_{A'})_{A' \neq A},~(Y_u)_{u \neq v}) \le p_A \cdot \prod_{u \in A \setminus \{v\}} \eps_u^{-1} .\]

    To make this conditional probability more concrete, note that if $U_n \uparrow V \setminus \{v\}$ and $\cB_n \uparrow \cP_{\text{fin}}(V) \setminus \{A\}$, where $\cP_{\text{fin}}(V)$ is the collection of finite subsets of $V$, then the above conditional probability is almost surely the limit of $\P(X_A=1 \mid (X_{A'})_{A' \in \cB_n},~(Y_u)_{u \in U_n})$ as $n \to \infty$. Thus, the desired bound will follow by showing that for any two finite disjoint subsets $\cI,\cJ \subset \cP_{\text{fin}}(V) \setminus \{A\}$ and any two finite disjoint subsets $I,J \subset V \setminus \{v\}$,
    \[ \P(X_A=1 \mid X_\cI \equiv 0,~ X_\cJ \equiv 1,~ Y_I \equiv 0,~ Y_J \equiv 1) \le p_A \cdot \prod_{u \in A \setminus \{v\}} \eps_u^{-1} .\]
    We may assume that $A \setminus \{v\} \subset I \cup J$.
    Since $\{X_A=1\} \subset \{\bar X_A \equiv 1\} \subset \{ Y_A \equiv 1\}$, we may also assume that $A \cap I = \emptyset$, so that $A \setminus \{v\} \subset J$. Since $\{ Y_I \equiv 0\} = \{ \bar X_I \equiv 0\} \cap \{\xi_I \equiv 0\}$, and $\xi$ is independent of $(X,\bar X,Y)$, we may replace the conditioning on $Y_I \equiv 0$ with $\bar X_I \equiv 0$. This, in turn, is simply the event that $X_{A'}=0$ for all $A'$ that intersect $I$, and we can incorporate this in the conditioning on $X_\cI \equiv 0$ by adding these sets to $\cI$ (note that this makes $\cI$ infinite, but this causes no trouble). Our goal is therefore to show that
    \[ \P(X_A=1 \mid X_\cI \equiv 0,~ X_\cJ \equiv 1,~ Y_J \equiv 1) \le p_A \cdot \prod_{u \in A \setminus \{v\}} \eps_u^{-1} .\]
    Using that $\P(X_A=1 \mid X_\cI \equiv 0,~ X_\cJ \equiv 1)=p_A$ and the fact that $X_A=1$ implies $Y_A \equiv 1$, this is the same as
    \[ \frac{\P(Y_{J \setminus A} \equiv 1 \mid X_\cI \equiv 0,~ X_\cJ \equiv 1,~X_A=1)}{\P( Y_J \equiv 1 \mid X_\cI \equiv 0,~ X_\cJ \equiv 1)} \le \prod_{u \in A \setminus \{v\}} \eps_u^{-1} .\]
    Since $Y_{J \setminus A}$ is conditionally independent of $X_A$ given $\{X_\cI \equiv 0,~ X_\cJ \equiv 1\}$, and $J \cap A = A \setminus \{v\}$, we see that
    \[ \frac{\P( Y_J \equiv 1 \mid X_\cI \equiv 0,~ X_\cJ \equiv 1)}{\P( Y_{J \setminus A} \equiv 1 \mid X_\cI \equiv 0,~ X_\cJ \equiv 1,~ X_A=1)} = \P( Y_{A \setminus \{v\}} \equiv 1 \mid  Y_{J \setminus A} \equiv 1,~ X_\cI \equiv 0,~ X_\cJ \equiv 1) .\]
    Finally, $\{ Y_{A \setminus \{v\}} \equiv 1\}$ contains the event $\{\xi_{A \setminus \{v\}} \equiv 1\}$, which is independent of $X$ and $Y_{J \setminus A}$, so that
    \[ \P( Y_{A \setminus \{v\}} \equiv 1 \mid  Y_{J \setminus A} \equiv 1,~ X_\cI \equiv 0,~ X_\cJ \equiv 1) \ge \P(\xi_{A \setminus \{v\}} \equiv 1) = \prod_{u \in A \setminus \{v\}} \eps_u . \qedhere \]
\end{proof}

\begin{remark}\label{rem:stoc-dom-lemma-sequential}
    There is a version of \cref{lem:domination3} in which $Y$ is exposed sequentially. Specifically, let $\preceq$ be a total order on $V$, let $\eps_v$ be arbitrary, and define $\delta_v := \epsilon_v + (1-\eps_v)\sum_{A \ni v} p_A \cdot \prod_{u \in A, u \prec v} \eps_u^{-1}$. Let $\xi$, $Y$, $Z$ be as before. Then minor modifications to the proof show that
    $\P(Y_v=1 \mid (X_A)_{A \not\ni v},~ (Y_u)_{u \prec v}) \le \delta_v$ almost surely for all $v\in V$,
    which implies that $Y$ is stochastically dominated by $Z$.
\end{remark}

\begin{proof}[Proof of \cref{thm:stochastic-domination}, first part]
We first consider the case when $V=\Gamma=\Z^d$.
Let $\lambda>0$ and $C>0$ be such that $\sum_{A \ni v} p_A e^{\lambda |A|} \le C$.
Let $N$ be large enough so that
\[ \sum_{A \ni v, \diam A > N} p_A e^{\lambda |A|} \le 1 .\]
We decompose our process $\bar X$ into the maximum of two (independent) processes, one corresponding to sets of diameter at most $N$, and one corresponding to the remaining sets. It suffices to show that each of the two processes is stochastically dominated by a non-trivial IID process.

Consider first the process $\bar X^1$ corresponding to $(p_A)_{A \in \cA}$, where $\cA := \{ A : \diam A \le N \}$.
By this we mean the process defined by
\[ \bar X^1 := \bigcup \{ A \in \cA : X_A=1 \} .\]
Let $(Y_{A,v})_{A \in \cA, v \in A}$ be independent Bernoulli random variables, with $\E Y_{A,v} = p_A^{1/|A|}$.
Note that we may couple everything so that $X_A = \prod_{v \in A} Y_{A,v}$ almost surely for all $A \in \cA$.
Hence, for any $u$,
\[ \bar X^1_u = \max_{A \in \cA, u \in A} X_A = \max_{A \in \cA, u \in A} \prod_{v \in A} Y_{A,v} \le \max_{A \in \cA, u \in A} Y_{A,u} =: Z_u .\]
In particular, $\bar X^1$ is stochastically dominated by the IID process $Z=(Z_u)$.
Finally, this IID process is non-trivial since $\{ A \in \cA : u \in A \}$ is finite, and each $Y_{A,u}$ has positive probability to be zero.

Consider next the process $\bar X^2$ corresponding to $(p_A)_{A \in \cA'}$, where $\cA' := \{ A : \diam A > N \}$.
By \cref{lem:domination2}, taking $\eps_v := e^{-\lambda}$ for all $v$, we see that this process is stochastically dominated by an IID process of density $e^{-\lambda} + e^{-\lambda}(1-e^{-\lambda})\sum_{A \in \cA', v \in A} p_A e^{\lambda|A|} \le e^{-\lambda}(2-e^{-\lambda}) < 1$.
This completes the proof in the case when $V=\Gamma=\Z^d$.

Consider now the general case.
Observe that our only use of the $\Z^d$-structure in the above proof is in order to obtain a suitable decomposition $\cA$ and $\cA'$ of the finite subsets of $\Z^d$. The required properties are that $\cA$ and $\cA'$ are $\Gamma$-invariant (so that $\bar X^1$ and $\bar X^2$ are $\Gamma$-invariant processes), that $\cA$ contains only finitely many sets containing a fixed $v \in V$, and that $\sum_{A \in \cA', v \in A} p_A e^{\lambda|A|} \le 1$.
The same proof works, for example, whenever $\Gamma$ is a subgroup of the automorphism group of a locally finite graph on $V$. For the general case, we need only show the existence of a decomposition with the same three properties. To this end, for $A \Subset V$, let $[A]:=\{\gamma(A) : \gamma \in \Gamma\}$ be the orbit of $A$ under $\Gamma$, and denote $p_{[A]}:=p_A$, which is well defined by the assumption that $(p_A)$ is invariant under $\Gamma$. Let $\{\cA_i\}_{i=1}^\infty$ be an enumeration of all orbits $[A]$ for which $p_{[A]} > 0$. Note that this is a $\Gamma$-invariant partition of the (allowed) finite subsets of $V$. Each $\cA_i$ contains only finitely many sets containing a fixed $v \in V$, since $\sum_{A \in \cA_i, v\in A} p_A = p_{[A]} \cdot |\{A \in \cA_i : v\in A\}|$ is finite. We may now choose $N$ large enough so that the decomposition $\cA:=\cA_1\cup\cdots\cup\cA_N$ and $\cA':=\cA_{N+1}\cup\cA_{N+2}\cup\cdots$ satisfies the required properties.
\end{proof}

\begin{proof}[Proof of \cref{thm:stochastic-domination}, second part]
Suppose that $V=\Gamma=\Z^d$. Denote $p_n := \sum_{A \ni 0, |A|=n} p_A$.
We shall show that for every $n \ge 1$ and $\beta \ge 10$, there exists a finite set $S \subset \Z^d$ such that
\begin{equation}\label{eq:dom-S-lower-bound}
\P(\bar X_S \equiv 1)^{1/|S|} \ge \exp\left(-\frac{\frac{8 \beta}n \log\big(1+\frac n{p_n}\big) + C\beta e^{-\beta}}{1-8e^{-\beta}} \right) .
\end{equation}
Note that the existence of such a set implies that $\bar X$ cannot be stochastically dominated by an IID process whose density is strictly less than the right-hand side. When~\eqref{eq:dom-exp-moment} does not hold, the right-hand side can be made arbitrarily close to 1, by first taking $n$ to infinity along a subsequence such that $p_n = e^{-o(n)}$, and then taking $\beta$ to infinity. This yields the theorem.

Towards showing~\eqref{eq:dom-S-lower-bound}, fix $n \ge 1$ and $\beta \ge 10$.
For each $b \in \Z^d$, let $Y_b$ be an independent random subset of $\Z^d$ with $\P(Y_b=A) \propto p_A \1_{\{|A|=n, \min A=b\}}$, where we write $\min A$ for the coordinate-wise minimum of the elements in $A$.
For $B \subset \Z^d$, define $Y_B := \bigcup_{b \in B} Y_b$.
We claim that for any $N$ large enough there exists a set $B$ such that
\begin{equation}\label{eq:dom-B}
 |B| \le \frac{8\beta N^d}{n} \qquad\text{and}\qquad \E |[N]^d \setminus Y_B| \le 4N^d e^{-\beta} .
\end{equation}
To see this, let ${\sf B}$ be a random set, independent of $(Y_v)$, with each $b \in [N]^d$ included independently with probability $\frac{2\beta}n$.
For $v \in [N]^d$,
\[ \P(v \notin Y_{{\sf B}}) = \prod_{b \in [N]^d} \P(b \notin {\sf B}\text{ or }v \notin Y_b) = \prod_{b \in [N]^d} ( 1 - \tfrac {2\beta} n\P(v \in Y_b)) \le \exp\left(-\frac {2\beta} n \sum_{b \in [N]^d} \P(v \in Y_b)\right) .\]
We now need to lower bound the sum in the exponent.
Note that if the sum over $b$ was not truncated to $[N]^d$, this would be simple as $\sum_{b \in \Z^d} \P(v \in Y_b) = \sum_{v \in [N]^d} \P(v \in Y_b) = \E |Y_b| = n$. We now argue that for most $v \in [N]^d$, the truncated sum is not much smaller. Indeed,
\[ \sum_{b \in [N]^d} \P(v \in Y_b) = \sum_{b \in [N]^d} \P(v-b \in Y_0) = \E |Y_0 \cap (v-[N]^d)| \ge \E |Y_0 \cap [n_1]^d| \ge \frac n2 ,\]
with the first inequality holding whenever all coordinates of $v$ are at least $n_1$, and the second inequality holding when $n_1$ is a large enough constant (depending on $n$, but not on $N$). Plugging this lower bound into the previous bound on $\P(v \notin Y_{{\sf B}})$, and summing over $v$, we obtain for $N$ large enough that
\[ \E |[N]^d \setminus Y_{\sf B}| = \sum_{v \in [N]^d} \P(v \notin Y_{{\sf B}}) \le d n_1 N^{d-1} + e^{-\beta} N^d \le 2e^{-\beta} N^d .\]
Thus, by Markov's inequality, with probability at least $\frac 12$, we have that $\E[|[N]^d \setminus Y_{\sf B}| \mid {\sf B}] \le 4e^{-\beta} N^d$.
Since $\E |{\sf B}| = \frac{2\beta}n N^d$, we have that $|{\sf B}| \le \frac{8\beta}n N^d$ with probability at least $\frac 34$.
Thus, with probability at least $\frac 14$, both hold simultaneously. In particular, there exists a choice of $B$ which satisfies~\eqref{eq:dom-B}.

Fix $B$ as in~\eqref{eq:dom-B}. By Markov's inequality,
\[ \P\left(|[N]^d \setminus Y_B| \le 8 e^{-\beta} N^d\right) \ge \tfrac12 .\]
Since this probability is the sum of $\P([N]^d \setminus Y_b =D)$ over all subsets $D$ of size at most $8e^{-\beta} N^d$, and since there are at most $\frac12 e^{C \beta e^{-\beta} N^d}$ such sets, there exists a set $D$ such that
\begin{equation}\label{eq:dom-D}
 |D| \le 8e^{-\beta}N^d \qquad\text{and}\qquad \P([N]^d \setminus Y_B \subset D) \ge e^{-C\beta e^{-\beta} N^d} .
\end{equation}

We now lower bound the probability that $\bar X_{[N]^d \setminus D} \equiv 1$.
To this end, for each $b \in \Z^d$, let $Z_b$ be an independent random subset of $\Z^d$ obtained by including each $A$ independently with probability $p_A \1_{\{|A|=n,\min A=b\}}$. For $U \subset \Z^d$, denote $Z_U := \bigcup_{v \in U} Z_v$.
Observe that $\P(Z_b \in \cdot \mid Z_b \neq \emptyset)$ stochastically dominates $Y_b$. Observe also that $\bigcup \{ A : X_A=1,~|A|=n \}$ and $Z_{\Z^d}$ have the same distribution. Thus,
\begin{align*}
 \P(\bar X_{[N]^d \setminus D} \equiv 1)
  &\ge \P([N]^d \setminus D \subset Z_{\Z^d}) \\
  &\ge \P([N]^d \setminus D \subset Z_B) \\
  &= \P([N]^d \setminus Z_B \subset D) \\
  &\ge \P([N]^d \setminus Z_B \subset D \mid Z_b \neq \emptyset \text{ for all }b \in B) \cdot \P(Z_b \neq \emptyset \text{ for all }b \in B) \\
  &\ge \P([N]^d \setminus Y_B \subset D) \cdot \P(Z_0 \neq \emptyset)^{|B|} .
\end{align*}
By Paley--Zygmund's inequality, letting $\cZ := |\{ A : |A|=n,~\min A =0,~X_A=1\}|$,
\[ \P(Z_0 \neq \emptyset) = \P(\cZ>0) \ge \frac{(\E \cZ)^2}{\E[\cZ^2]} = \frac{(\E \cZ)^2}{\text{Var}(\cZ)+(\E \cZ)^2} \ge \frac{\E \cZ}{1+\E \cZ} = \frac{p_n/n}{1+p_n/n} = \frac{p_n}{n+p_n} .\]
Using~\eqref{eq:dom-B} and~\eqref{eq:dom-D}, we obtain that
\begin{equation}\label{eq:streak-lower-bound}
 \P(\bar X_{[N]^d \setminus D} \equiv 1) \ge \exp\left(-\tfrac{8\beta}n N^d \log\big(1+\tfrac n{p_n}\big) - C \beta e^{-\beta} N^d\right) .
\end{equation}
Finally, using the bound on the size of $D$ from~\eqref{eq:dom-D}, we see that~\eqref{eq:dom-S-lower-bound} holds with $S:=[N]^d \setminus D$.
\end{proof}

\begin{remark}\label{rem:optimal-dom-in-terms-of-exp-moment}
The proof of the second part of \cref{thm:stochastic-domination} gives an effective bound on the density of an IID process that stochastically dominates $\bar X$ in terms of the best exponential moment.
Denote $p_n := \sum_{A \ni 0, |A|=n} p_A$ as before.
Define
\[ \delta := \sup_{S \Subset \Z^d} \P(\bar X_S \equiv 1)^{\frac 1{|S|}}, \qquad \lambda_c := \limsup_{n \to \infty} \frac{-\log p_n}{n}, \qquad \gamma := \inf_{n \ge 1} \frac1n \log\Big(1 + \frac n{p_n}\Big) .\]
Note that $\bar X$ cannot be stochastically dominated by an IID process of density less than $\delta$.
Note also that $\sum_{A \ni 0} p_A e^{\lambda |A|}$ is finite for all $\lambda<\lambda_c$ and is infinite for all $\lambda>\lambda_c$, and that $\gamma \le \lambda_c$.
It follows from~\eqref{eq:dom-S-lower-bound} that
\begin{equation}\label{eq:dom-delta-beta-gamma-bound}
\log \tfrac1\delta \le \inf_{\beta \ge 10} \frac{8 \beta \gamma + C\beta e^{-\beta}}{1-8e^{-\beta}} .
\end{equation}
Taking $\beta = \max\{10, \log \frac1\gamma \}$, we see that
\begin{equation}\label{eq:dom-delta-gamma-bound}
 \log \tfrac1\delta \le C\gamma \max\{1,\log \tfrac1\gamma\} .
\end{equation}
In particular, this implies that $\bar X$ cannot be stochastically dominated by an IID process of density less than $\exp(-C\lambda_c \max\{1,\log \tfrac1{\lambda_c}\})$.
\end{remark}

\begin{remark}
    Without an invariance assumption in \cref{thm:stochastic-domination}, the exponential moment condition~\eqref{eq:dom-exp-moment} is not necessary, as the following example demonstrates. Fix $v \in V$ and let $\{A_{n,i}\}_{n,i \ge 1}$ be pairwise disjoint sets with $v \notin A_{n,i}$ and $|A_{n,i}|=n-1$. Set $p_A := \frac1{8n^2} 4^{-n}$ if $A = A_{n,i} \cup \{v\}$ for some $(n,i)$ such that $i \le 4^n$, and set $p_A:=0$ otherwise. Then, for all $\lambda>0$,
    \[ \sum_{A : v \in A} p_A e^{\lambda |A|} = \sum_{n=1}^\infty \tfrac1{8n^2} e^{\lambda n} = \infty .\]
    On the other hand, using that $\P(\bar X_v=1) \le \sum_{A:v \in A} p_A = \sum_n \frac1{8n^2} \le \frac14$, it is not hard to show (via an argument as in the proof of \cref{lem:domination2}) that $\bar X$ is stochastically dominated by an IID Bernoulli process of parameter $\frac12$.
\end{remark}

\section{Finitary coding}\label{sec:finitary-coding}

In this section we prove the results about finitary codings.
We start in \cref{sec:no-finitary-coding} with the proof of \cref{thm:ffiid-implies-some-exp-moment} which states that there cannot be a finitary coding from IID without an exponential moment.
We then move on to the proof of \cref{thm:ffiid-under-all-exp-moment}. We first give a proof in \cref{sec:pairs} of the simpler case when only pairs are allowed, i.e., $p_A=0$ unless $|A|=2$. We then give the proof of the general case in \cref{sec:finitary-coding-all-exp-moments}.
We next give the proof of \cref{thm:ffiid-with-large-p1} in \cref{sec:ffiid-p1}.
We then turn to the proof of \cref{thm:ffiid-connected}. We start in \cref{sec:connected} with the proof of the general case, but in the case of $\Z$ only obtain the result with $\lambda_2>0$ rather than the stated $\lambda_2=0$. Finally, in \cref{sec:ffiid-one-dim} we consider the one dimensional case in more depth, completing the proof that one can take $\lambda_2=0$ in \cref{thm:ffiid-connected}, and further proving a finitary isomorphism result for a related process.

\subsection{No finitary coding without an exponential moment}\label{sec:no-finitary-coding}

\begin{proof}[Proof of \cref{thm:ffiid-implies-some-exp-moment}]
    Suppose that~\eqref{eq:exp-moment} does not hold, so that all exponential moments are infinite.
    Denote $p_n := \sum_{A \ni 0, |A|=n} p_A$. We have seen in the proof of \cref{thm:stochastic-domination} (see~\eqref{eq:dom-D} and~\eqref{eq:streak-lower-bound}) that for any $n \ge 1$, $\beta \ge 10$ and $N$ large enough, there exists a set $D \subset [N]^d$ such that
    \[ |D| \le 8 e^{-\beta} N^d \qquad\text{and}\qquad \P(\bar X_{[N]^d \setminus D} \equiv 1) \ge \exp\left(-\tfrac{8\beta}n N^d \log\big(1+\tfrac n{p_n}\big) - C \beta e^{-\beta} N^d\right) .\]
    Since $\bar X$ is positively associated~\cite[Theorem 2.4]{forsstrom2024poisson}, writing $\P(\bar X_v=1)=e^{-a}$ with $a \in (0,\infty)$, we obtain that
    \[ \P(\bar X_{[N]^d} \equiv 1) \ge \P(\bar X_{[N]^d \setminus D} \equiv 1) \cdot \P(\bar X_D \equiv 1) \ge e^{-\big(\tfrac{8\beta}n \log(1+\tfrac n{p_n}) + C\beta e^{-\beta}+ 8a e^{-\beta}\big)N^d} .\]
    Note that $\lim_{n\to\infty} \frac1n \log(1+\frac n{p_n})=0$ since all exponential moments are infinite. Thus, the constant in the exponent can be made arbitrarily small, and we see that
    \[ \P(\bar X_{[N]^d} \equiv 1) \ge e^{-o(N^d)} \qquad\text{as }N \to \infty .\]
    This rules out the possibility that $\bar X$ is FFIID, since if it was, it would satisfy the mean ergodic theorem with an exponential rate~\cite{bosco2010exponential}, and, in particular, $\P(\bar X_{[N]^d} \equiv 1) \le e^{-\Omega(N^d)}$, since the density of $\bar X$ is strictly less than 1.
\end{proof}

\subsection{The case of pairs (a special case of \cref{thm:ffiid-under-all-exp-moment})}\label{sec:pairs}

Here we prove the special case of \cref{thm:ffiid-under-all-exp-moment} in which we assume that $p_A = 0$ unless $|A|=2$.
We present the proof in one dimension.
Denote
\[ p_n := p_{\{0,n\}} .\]
Recall that there is a simple finitary coding in the case when $\sum \sqrt{p_n} < \infty$ by ``splitting'' the randomness associated to any given $X_{\{i,i+n\}}$ into each of its endpoints $i$ and $i+n$. More generally, the idea is to group together pairs with the same minimal element and roughly comparable maximal element, and handle the splitting for all pairs in any such group simultaneously (with each such group handled separately). For this, we shall require a special case of the sequential version of \cref{lem:domination3} (see \cref{rem:stoc-dom-lemma-sequential}), which we state in the following lemma.

\begin{lemma}\label{lem:domination}
    Let $X_1,\dots,X_n$ be independent Bernoulli random variables and denote $X_0 := \max\{X_1,\dots,X_n\}$. Let $\epsilon>0$. Then $(X_0,\dots,X_n)$ is stochastically dominated by $(Z_0,\dots,Z_n)$, where the latter are independent Bernoulli random variables with $\E Z_0 \le \E X_0 + \epsilon$ and $\E Z_i \le \frac1\epsilon \E X_i$ for $1 \le i \le n$.
\end{lemma}

\begin{proof}[Proof of \cref{thm:ffiid-under-all-exp-moment} in the special case of pairs]
Let $N_k$ be the smallest positive integer such that
\[ \sum_{n=N_k}^\infty p_n \le k^{-4} .\]
Define $I_k := [N_k, N_{k+1}) \cap \Z$ and
\[ \bar X^k := \bigcup \{ \{i,j\} : X_{\{i,j\}}=1,~|i-j| \in I_k \} .\]
Clearly,
\[ \left(\max_{n \in I_k} \{X_{\{i,i+n\}},X_{\{i,i-n\}}\}\right)_{i \in \Z} \text{ has the same law as }\bar X^k .\]
This expresses $\bar X^k$ as a block factor of IID, but of course in this way, the coding radius is never zero (assuming $p_n>0$ for some $n \in I_k$). Our goal is to express $\bar X^k$ differently so as to have this additional property. Specifically, the coding radius will be positive with probability at most $3k^{-2}$. Since this is summable, when we consider the block codes for $\bar X^k$ for all $k$, only finitely many of them will have a positive coding radius. This means that $\bar X = (\sup_k \bar X^k_i)_{i \in \Z}$ is FFIID.

We start by focusing on the pairs whose minimal coordinate is zero, namely, $\{\{0,n\}: n\in I_k\}$.
For brevity, denote $I:=I_k$, $X_n := X_{\{0,n\}}$ for $n \in I$, and $X_0 := \max_{n \in I} X_n$.
Let $W_0$ and $(W_n)_{n \in I}$ be independent Bernoulli random variables with $\E W_0 = 2k^{-2}$ and $\E W_n = k^2 p_n$.
By \cref{lem:domination}, $(X_0,(X_n)_{n \in I})$ is stochastically dominated by $(W_0,(W_n)_{n \in I})$. Let $\pi$ be a monotone coupling between $(X_0,(X_n)_{n \in I})$ and $(W_0,(W_n)_{n \in I})$. We view $\pi$ as a probability measure on $(\{0,1\}^{\{0\} \cup I})^2$.

For each $w \in \{0,1\}^{\{0\} \cup I}$, let $\pi_w$ be the conditional distribution of $(X_0,(X_n)_{n \in I})$ under $\pi$ given that $(W_0,(W_n)_{n \in I})=w$.
Let $\varphi_w \colon [0,1] \to \{0,1\}^{\{0\} \cup I}$ be a measurable function that pushes forward the Lebesgue measure on $[0,1]$ to $\pi_w$. Put differently, if $U$ is a uniform random variable on $[0,1]$, then $\varphi_w(U)$ has law $\pi_w$. Thus, writing $\varphi(u,w) := \varphi_w(u)$, we have that $\varphi(U,W_0,(W_n)_{n \in I})$ (dropping some parenthesis for clarity) has the same distribution as $(X_0,(X_n)_{n \in I})$ (assuming $U$ and $W$ are independent).

We now put this together for all $i \in \Z$ simultaneously.
Let $(W_{i,n})_{i \in \Z, n \in \{0\}\cup I}$ be independent Bernoulli random variables with $\E W_{i,0} = 2k^{-2}$ and $\E W_{i,n} = k^2 p_n$.
Let $(U_i)_{i \in \Z}$ be an independent IID process of uniform random variables.
For $i \in \Z$ and $n \in I$, define
\[ X^{+n}_i := \varphi(U_i, W_{i,0}, (W_{i+m,m})_{m \in I})_n .\]
Observe that $(X^{+n}_i)_{i \in \Z, n \in I}$ has the same law as $(X_{\{i,i+n\}})_{i \in \Z, n \in I}$. Thus, denoting $X^{-n}_i:=X^{+n}_{i-n}$, we have as before that
\[ \left(\max_{n \in I} \{X^{+n}_i,X^{-n}_i\}\right)_{i \in \Z} \text{ has the same law as }\bar X^k ,\]
and this expresses $\bar X^k$ as a block factor of the IID process $(U,W)$.

We now show that this block factor has the additional property that the coding radius is typically zero. Specifically, we claim that the coding radius is zero on the event $\{W_{i,0}=0, W_{i,I} \equiv 0 \}$, which is clearly an event of probability at least $1-3k^{-2}$. Even more specifically, we claim that for any $n \in I$,
\[ W_{i,n} = 0 \quad\implies\quad X^{-n}_i =0 \]
and
\[ W_{i,0} = 0 \quad\implies\quad X^{+n}_i =0 .\]
Indeed, the first claim follows from the fact that almost surely $\varphi_w(U)_n \le w_n$, and the second claim from the fact that almost surely $\varphi_w(U)_n \le w_0$.
\end{proof}

\subsection{The case of all exponential exponents (\cref{thm:ffiid-under-all-exp-moment})}\label{sec:finitary-coding-all-exp-moments}

Here we prove the general case of \cref{thm:ffiid-under-all-exp-moment}. Thus, we allow for $p_A$ to be non-zero for arbitrary $A$ and only assume the existence of all exponential moments as in~\eqref{eq:exp-moment-all}.

\begin{proof}[Proof of \cref{thm:ffiid-under-all-exp-moment}]
Denote $N_0:=0$. For $k \ge 1$, let $N_k$ be the smallest non-negative integer such that
\[ \sum_{A \Subset \Z^d : 0 \in A, \diam A \ge N_k} p_A k^{2|A|} \le 1 .\]
For $k \ge 0$, define
\[ \cA_k := \{ A \Subset \Z^d : N_k \le \diam A < N_{k+1} \} .\]
and
\[ \bar X^k := \bigcup \{ A \in \cA_k : X_A=1 \} .\]
Note that $\bar X = \bar X^0 \cup \bar X^1 \cup \cdots$.
Clearly, each $\bar X^k$ is a block factor of IID.
We show that, for any $k \ge 1$, we may express $\bar X^k$ as a finitary factor of IID with the property that its coding radius is positive with probability at most $1/k^2$.
This implies the theorem.

For a set $A \Subset \Z^d$, write $\min A$ for its lexicographical minimum element.
We start by focusing on sets whose minimum belongs to some given set $L \Subset \Z^d$ with $\min L = 0$, namely, $\cA_{k,L} := \{ A \in \cA_k : \min A \in L \}$. Write $L^*$ for the ball of radius $N_{k+1}$ around $L$, and note that every $A \in \cA_{k,L}$ satisfies $A \subset L^*$.
For $v \in L^*$, define
\[ Y_v := \max \{ X_A : A \in \cA_{k,L},~ v \in A \} .\]
Let $W=(W_v)_{v \in L^*}$ be IID Bernoulli random variables with parameter $1/k^2$.
By \cref{lem:domination2}, $Y=(Y_v)_{v \in L^*}$ is stochastically dominated by $W$.
In particular, there is a monotone coupling between the two.
Since $Y$ is a function of $X^L := (X_A)_{A \in \cA_{k,L}}$, there is a coupling of $X^L$ and $W$ under which $Y \le W$ almost surely.
Let $\pi_L$ be such a coupling.
We view $\pi_L$ as a probability measure on $\{0,1\}^{\cA_{k,L}} \times \{0,1\}^{L^*}$.
For each $w \in \{0,1\}^{L^*}$, let $\pi_{L,w}$ be the conditional distribution of $X^L$ under $\pi_L$ given that $W=w$.
Let $\varphi_{L,w} \colon [0,1] \to \{0,1\}^{\cA_{k,L}}$ be a measurable function that pushes forward the Lebesgue measure on $[0,1]$ to $\pi_{L,w}$. Put differently, if $U$ is a uniform random variable on $[0,1]$, then $\varphi_{L,w}(U)$ has law $\pi_{L,w}$. Thus, writing $\varphi_L(u,w) := \varphi_{L,w}(u)$, we have that $\varphi_L(U,W)$ has the same distribution as $X^L$ (assuming $U$ and $W$ are independent).
This defines $\varphi_L$ for all $L$ having $\min L = 0$.
For arbitrary $L \Subset \Z^d$, we define $\varphi_L \colon [0,1] \times \{0,1\}^{L^*} \to \{0,1\}^{\cA_{k,L}}$ by $\varphi_L(u,\cdot) = T_{\min L} \circ \varphi_{L-\min L}(u,\cdot) \circ T_{\min L}^{-1}$, where $T_v$ denotes translation by $v$.
This will ensure that our construction is equivariant.

We now put this together for all $v \in \Z^d$ simultaneously.
Fix a large integer $\ell$.
We construct $\bar X^k$ as a finitary factor of the IID process $(U,V,W)$, where $U$ and $V$ are IID processes of uniform random variables, $W=(W_{v,l})_{v \in \Z^d, 1 \le l \le \ell}$ is an IID process of Bernoulli random variables with parameter $k^{-2}$, and the three processes $U,V,W$ are independent.
Let $M$ be a random partition of $\Z^d$ into finite partition classes, such that $M$ is a finitary factor of $V$ and such that any ball of radius $N_{k+1}$ intersects at most $\ell$ partition classes of $M$ almost surely (this is where the choice of $\ell$ is made). The existence of such a random partition is shown in \cref{lem:partition} below.

For $v \in \Z^d$, let $L_v$ denote the partition class containing $v$ and let $M_v := \min L_v$.
Define $\cM := \{ M_v : v \in \Z^d \}$ and $\cM_v := \{ u \in \cM: v \in L_u^* \}$.
Note that $M_v \in \cM_v$ and $|\cM_v| \le \ell$ almost surely.
For $u \in \cM$ and $v \in L_u^*$, let $W^u_v$ equal $W_{v,j}$ if $u$ is the $j$-th smallest element of $\cM_v$ in lexicographical order. For $u \in \cM$, define $W^u := (W^u_v)_{v \in L_u^*}$ and
\[ \cX^u := \varphi_{L_u}(U_u, W^u) \in \{0,1\}^{\cA_{k,L_u}} .\]
For $A \in \cA_k$, define
\[ \cX_A := \cX^{\min A}_A .\]

We claim that $(\cX_A)_{A \Subset \Z^d}$ has the same law as $(X_A)_{A \in \cA_k}$, i.e., that of independent Bernoulli random variables with parameters $(p_A)_{A \in \cA_k}$.
In fact, we claim that this is true even conditionally on $V$. To this end, we condition on $V$.
Now note that $(\cX^u)_{u \in \cM}$ are independent. Thus, it suffices to show that each $\cX^u$ has the same distribution as $X^{L_u}=(X_A)_{A \in \cA_{k,L_u}}$, which follows from the choice of $\varphi_L$.

It is now not hard to conclude the proof.
Indeed,
\[ \bar \cX := \bigcup \{ A \in \cA_k : \cX_A = 1 \} ,\]
has the same distribution as $\bar X$, and it is clear from the definitions that it is a finitary factor of $(U,V,W)$. It therefore remains only to show that the coding radius is typically zero. Specifically, we claim that the coding radius is zero on the event that $W_{0,1}=\cdots=W_{0,\ell}=0$. Even more specifically, we claim that $\bar \cX_0=0$ on this event. Indeed, this follows from the fact that $\max\{\varphi_{L,w}(U)_A : A \in \cA_{k,L}, v \in A\} = 0$ (almost surely) whenever $w_v=0$.
\end{proof}

It remains to prove the following lemma which we used in the construction above.
\begin{lemma}\label{lem:partition}
    For every $r \ge 1$, there exists a FFIID partition of $\Z^d$ into finite partition classes such that every ball of radius $r$ in $\Z^d$ intersects at most $5^d$ partition classes almost surely.
\end{lemma}
\begin{proof}
    It suffices to prove the lemma for even $r$ and for the $\ell_\infty$ metric.
    Suppose that $\cZ$ is a FFIID subset of $\Z^d$ that is almost surely an $(r,r)$-net, that is, any two points in $\cZ$ are at distance greater than $r$, and any vertex in $\Z^d$ is at distance at most $r$ from some point in $\cZ$. Let $M$ be the partition of $\Z^d$ into the Voronoi cells of $\cZ$, breaking ties according to another independent IID process. More explicitly, let $U$ be an independent IID process of uniform random variables in $[0,1]$, and define the Voronoi cell $M_v$ of a point $v \in \cZ$ to consist of all $u \in \Z^d$ for which $\dist(u,v)=\dist(u,\cZ)$ and $U_v \le U_w$ for any $w \in \cZ$ such that $\dist(u,w)=\dist(u,\cZ)$.

    We claim that the random partition $M$ satisfies the desired properties. Clearly, it is a finitary factor of IID.
    Now consider a ball $B=B_r(u)$. We need to show that $B \cap M_v \neq \emptyset$ for at most $5^d$ points $v \in \cZ$. Let $V$ be the set of such $v$. For every $v \in V$, there is some $w \in B$ such that $\dist(w,v)=\dist(w,\cZ) \le r$. Thus, $V \subset B_{2r}(u)$.
    Since $\dist(v,v') > r$ for any two distinct $v,v' \in V$, we see that $\{B_{r/2}(v)\}_{v \in V}$ are pairwise disjoint sets contained in $B_{2.5r}(u)$. Thus, their total volume is at most that of $B_{2.5r}(u)$. We conclude that $|V| \le |B_{2.5r}(u)| / |B_{0.5r}(v)| = (5r+1)^d/(r+1)^d \le 5^d$.

    It remains to construct a random set $\cZ$ as above. This is rather standard. For example, we may take an IID process $W$ of uniform random variables and construct $\cZ$ is a greedy manner, by first taking all points whose $W$-value is minimal in the ball of radius $r$ around them, then removing from consideration all points at distance at most $r$ from those which were taken, and then again adding all points whose $W$-value is minimal in the ball of radius $r$ around them (from those which are still under consideration), and then removing from consideration all points at distance at most $r$ from those which were taken, and continuing in this manner. While this process never terminates (it continues for infinitely many steps), the state of any given vertex is almost surely determined after finitely many steps, so that this is a finitary factor. It remains to check that this yields a $(r,r)$-net almost surely. The fact that any two points in $\cZ$ are at distance at least $r$ is clear from the construction. The fact that any vertex in $\Z^d$ is at distance at most $r$ from some point in $\cZ$ follows from the easily verifiable fact that there are no infinite $r$-paths with descending $W$-values almost surely.
\end{proof}

\subsection{The large $p_{\{0\}}$ case (\cref{thm:ffiid-with-large-p1})}\label{sec:ffiid-p1}

In this section, we prove \cref{thm:ffiid-with-large-p1}.
Define a partial order $\preceq$ on $\{0,1,*\}$ in which $0 \prec *$ and $1\prec *$, but 0 and 1 are incomparable. This induces a pointwise partial order on $\{0,1,*\}^{\Z^d}$.
Note that there is a natural embedding of $\{0,1,*\}^L$ in $\{0,1,*\}^{\Z^d}$ by placing $*$ everywhere outside $L$. This allows to us to compare elements of $\{0,1,*\}^L$ and $\{0,1,*\}^{\Z^d}$.

\begin{lemma}
    Suppose that~\eqref{eq:exp-moment-p1} holds.
    Then there exists a translation-invariant family of random variables $Z^L \in \{0,1,*\}^L$, $L \subset \Z^d$, such that
    \begin{enumerate}
        \item $\bar X \preceq_{st} Z^M \preceq_{st} \le Z^L$ for all $L \subset M \subset \Z^d$.
        \item $(Z^L)_{L \subset \Lambda}$ and $(Z^L)_{L \subset \Z^d \setminus \Lambda}$ are independent for any $\Lambda \subset \Z^d$.
        \item $\P(Z^L_0=*) \to 0$ as $L \uparrow \Z^d$.
    \end{enumerate}
\end{lemma}
\begin{proof}
    Let $U=(U_v)_{v \in \Z^d}$ be independent uniform random variables on $[0,1]$, independent also of $(X_A)$.
    For $L \subset \Z^d$ and $v \in L$, define
    \[ Z^L_v := \begin{cases}
        * &\text{if }U_v \le \eps_{v,L} \\
        \max\{ X_A : A \Subset L,~ v \in A,~ X_A=1\} &\text{otherwise} \end{cases} ,\]
    where
    \[ \eps_{v,L} := \sum_{\substack{A \Subset \Z^d:\\v \in A,\, A \not\subset L}} \frac{p_A}{1-p_A} p_{\{0\}}^{-|A|} .\]
    This defines the family of random variables $(Z^L)$, which is clearly invariant to translations. Observe that $Z^L$ is measurable with respect to $(X_A)_{A \subset L}$ and $(U_v)_{v \in L}$. In particular, $(Z^L)_{L \subset \Lambda}$ and $(Z^L)_{L \cap \Lambda = \emptyset}$ are independent for any $\Lambda \subset \Z^d$.
    Finally, $\P(Z^L_v = *) = \eps_{v,L}$, which tends to zero as $L \uparrow \Z^d$ by the assumption~\eqref{eq:exp-moment-p1}.

    It remains to show that the first item holds, that is, that $\bar X \preceq_{st} Z^M \preceq_{st} Z^L$ when $L \subset M \subset \Z^d$. Note that $Z^{\Z^d} = \bar X$, so that it suffices to show that $Z^M \preceq_{st} Z^L$ when $L \subset M \subset \Z^d$. Fix $L$ and $M$.
    Define $\bar X^0 \in \{0,1\}^L$ and $\bar X^1 \in \{0,1\}^M$ by
    \begin{align*}
        \bar X^0_v &:= \max \{ X_A : v \in A,~ A \Subset L,~ |A|>1 \},\\
        \bar X^1_v &:= \max \{ X_A : v \in A,~ A \not\subset L,~ A \Subset M \}.
    \end{align*}
    Note that
    \begin{align*}
        Z^L_v &= \max\{\bar X^0_v, X_{\{v\}}\} &&\text{whenever }Z^L_v \neq *,~ v \in L,\\
        Z^M_v &= \max \{\bar X^0_v,\bar X^1_v, X_{\{v\}}\} &&\text{whenever }Z^M_v \neq *,~ v \in M.
    \end{align*}

    We shall prove the required stochastic domination conditionally on $\bar X^0$.
    In fact, we may condition on more, with no cost to the proof. Namely, let $\cF_L$ be the $\sigma$-algebra generated by $\{X_A : A \Subset L,|A|>1\}$. We show that $\P(Z^M \in \cdot \mid \cF_L) \preceq_{st} \P(Z^L \in \cdot \mid \cF_L)$ almost surely.

    Note that if $\bar X^0_v=1$, then $Z^M_v,Z^L_v \in \{1,*\}$ and $\{Z^M_v=*\} \subset \{Z^L_v=*\}$, so that $Z^M_v \preceq Z^L_v$. Thus, we only need to worry about comparing the restrictions of $Z^M$ and $Z^L$ to
    \[ I := \{ v \in L : \bar X^0_v=0 \} .\]
    For $v \in I$, we have that $Z^L_v = X_{\{v\}}$ whenever $Z^L_v \neq *$. Thus, the random variable $Z^L_v$ is very simple: it is a Bernoulli($p_{\{0\}}$) random variable, which is ``censored'' with probability $\eps_{v,L}$. More precisely, it equals $*$ with probability $\eps_{v,L}$, equals 1 with probability $p_{\{0\}}(1-\eps_{v,L})$, and equals 0 with probability $(1-p_{\{0\}})(1-\eps_{v,L})$. In addition, $\{Z^L_v\}_{v \in I}$ are conditionally IID given $\cF_L$. Something similar, but more complicated, holds for $Z^M$. We have that $Z^M_v = \max\{\bar X^1_v, X_{\{v\}}\}$ whenever $Z^M_v \neq *$, so that it is also a Bernoulli random variable (of parameter at least $p_{\{0\}}$), which is ``censored'' with probability $\eps_{v,M}$. However, there is dependence among the family $\{Z^M_v\}_{v \in I}$ coming through $\bar X^1$. It suffices to show that if we expose these random variables one-by-one, then at each step, conditionally on the previously exposed information, the random variable $Z^M_v$ under consideration is stochastically $\preceq$-dominated by $Z^L_v$ (recall that these are IID). Such stochastic domination amounts to bounding the probabilities of $Z^M_v$, showing that it is 1 with probability at least $p_{\{0\}}(1-\eps_{v,L})$, and that it is 0 with probability at least $(1-p_{\{0\}})(1-\eps_{v,L})$. We now proceed to do this, showing that this holds for each $v \in I$, even when conditioning on much more information.

    We now show that, almost surely, for $v \in I$,
    \[ \P(Z^M_v = 1 \mid \cF_L, (Z^M_u)_{u \neq v}) \ge p_{\{0\}}(1-\eps_{v,L}) .\]
    and
    \[ \P(Z^M_v = 0 \mid \cF_L, (Z^M_u)_{u \neq v}) \ge (1-p_{\{0\}})(1-\eps_{v,L}) \]
    The first inequality is immediate at the left-hand side is at least $\P(X_{\{v\}}=1, Z^M_v \neq *)=p_{\{0\}}(1-\eps_{v,M})$.
    For the second inequality, note that
    \begin{align*}
        \P(Z^M_v = 0 \mid \cF_L, (Z^M_u)_{u \neq v})
         &= \P(\bar X^1_v = X_{\{v\}}= 0,~Z^M_v \neq * \mid \cF_L, (Z^M_u)_{u \neq v}) \\
         &= (1-p_{\{0\}})(1-\eps_{v,M}) \cdot \P(\bar X^1_v = 0 \mid \cF_L, (Z^M_u)_{u \neq v}) ,
    \end{align*}
    so that the second inequality is equivalent to
    \[ \P(\bar X^1_v = 1 \mid \cF_L, (Z^M_u)_{u \neq v}) \le 1 - \frac{1-\eps_{v,L}}{1-\eps_{v,M}} .\]
    Note that the right-hand side is at least $\eps_{v,L} - \eps_{v,M}$, which in turn equals the sum of $p_A$ over all $A \Subset M$ having $v \in A$ and $A \not\subset L$. Thus, by a union bound, it suffices to show that for any $A \Subset M$ such that $A \not\subset L$, almost surely,
    \[ \P(X_A=1 \mid \cF_L, Z^M) \le \frac{p_A}{1-p_A} \cdot p_{\{0\}}^{-|A|} .\]
    Denote $\bar X^2 := (\max \{\bar X^1_u, X_{\{u\}}\})_{u \in M}$. We condition on more and show that
    \[ \P(X_A=1 \mid U,\bar X^2, (X_B)_{B \Subset \Z^d, B \neq A, |B|>1}) \le \frac{p_A}{1-p_A} \cdot p_{\{0\}}^{-|A|} .\]
    We may of course drop the conditioning on $U$ as it is independent of everything else.
    The probability in question is zero unless $\bar X^2_A \equiv 1$. In the latter case, letting $N$ denote the number of $u \in A$ for which $X_B=0$ for all $A \neq B \Subset \Z^d$ such that $u \in B$ and $|B|>1$, this probability is
    \[ \frac{p_A}{p_A + (1-p_A)p_{\{0\}}^N} \le \frac{p_A}{1-p_A} \cdot p_{\{0\}}^{-N} \le \frac{p_A}{1-p_A} \cdot p_{\{0\}}^{-|A|} ,\]
    which completes the proof of the lemma.
 \end{proof}

\begin{proof}[Proof of \cref{thm:ffiid-with-large-p1}]
    Given the previous lemma, the proof of the theorem is rather straightforward. The idea is to use $(Z^L)$ to independently sample portions of our process on a small scale, with some sites having an unspecified value $*$, and then keep moving to larger scales while refining our sampling each time relying on the stochastic $\preceq$-domination which guarantees that this is possible. Once the value at a site becomes specified, it never changes, so that procedure will indeed be finitary.

    Let $(Z^L)_{L \Subset \Z^d}$ be as in the previous lemma.
    Let $L \Subset \Z^d$ with $\min L = 0$ (where the minimum is, say, with respect to the lexicographical order), and let $L=L_1 \cup \cdots \cup L_n$ be a partition of $L$ into $n$ non-empty sets.
    Since $Z^L \preceq_{st} Z^{L_i}$ for each $1 \le i \le n$, and since $Z^{L_i}$ is specified only on $L_i$, it follows that $Z_L$ is stochastically $\preceq_{st}$-dominated by the concatenation $Z^{L_1,\dots,L_n}$ of the independent $Z^{L_i}, 1 \le i \le n$ (i.e., the element of $\{0,1,*\}^L$ that agrees with $Z^{L_i}$ on $L_i$). Let $\pi_{L,L_1,\dots,L_n}$ be a $\preceq$-monotone coupling between $Z_L$ and $Z^{L_1,\dots,L_n}$.
    For each $z \in \{0,1,*\}^L$, let $\pi_{L,L_1,\dots,L_n,z}$ be the conditional distribution of $Z^L$ under $\pi_L$ given that $Z^{L_1,\dots,L_n}=z$.
    Let $\varphi_{L,L_1,\dots,L_n,z} \colon [0,1] \to \{0,1,*\}^L$ be a measurable function that pushes forward the Lebesgue measure on $[0,1]$ to $\pi_{L,L_1,\dots,L_n,z}$. Put differently, if $U$ is a uniform random variable on $[0,1]$, then $\varphi_{L,L_1,\dots,L_n,z}(U)$ has law $\pi_{L,L_1,\dots,L_n,z}$. It should be understood that $\varphi_{L,L_1,\dots,L_n,z}$ does not depend on the ordering of $(L_i)$, but rather only on the partition $\{L_i\}$ of $L$ that it defines.
    We also write $\varphi_{L,L_1,\dots,L_n}(u,z) := \varphi_{L,L_1,\dots,L_n,z}(u)$.
    Finally, let $\phi \colon [0,1] \to \{0,1,*\}$ be such that $\phi(U)$ has the same law as $Z^{\{0\}}$.
    When $\min L \neq 0$, we define $\varphi_{L,L_1,\dots,L_n} \colon [0,1] \times \{0,1,*\}^L \to \{0,1,*\}^L$ by $\varphi_{L,L_1,\dots,L_n}(u,\cdot) = T_{\min L} \circ \varphi_{L-\min L}(u,\cdot) \circ T_{\min L}^{-1}$, where $T_v$ denotes translation by $v$.
    This will ensure that our construction is equivariant.

    We now show that $\bar X$ is FFIID.
    We use a so-called finitary hyperfinite exhaustion of $\Z^d$.
    This is a sequence $(P_n)$ of partitions of $\Z^d$, such that $P_{n-1}$ is a refinement of $P_n$, all partition classes of every $P_n$ are finite, any two vertices of $\Z^d$ eventually belong to the same partition class, and all $P_n$ are finitary factors of a common IID process $P$. It is standard that such a sequence exists. We may also assume that $P_0$ is the partition of $\Z^d$ into singletons.

    Now let $U=(U_{v,n})_{v \in \Z^d, n \ge 0}$ be an independent IID process of uniform random variables.
    We define a sequence of elements $(\cX^n)_{n=0}^\infty$ in $\{0,1,*\}^{\Z^d}$ as follows.
    First, define
    \[ \cX^0_v := \phi(U_v) .\]
    Next, for $n \ge 1$, we obtain $\cX^n$ from $\cX^{n-1}$ as follows.
    For $v \in \Z^d$, let $L^n_v$ denote the partition class of $P_n$ containing $v$, let $M^n_v := \min L^n_v$, and define $\cM^n := \{ M^n_v : v \in \Z^d \}$.
    For $u \in \cM^n$, define $\cX^n$ on $L_u$ by
    \[ \cX^n_{L_u} := \varphi_{L^n_u,\{L^{n-1}_v : v \in L^n_u\}}(U_u,\cX^{n-1}_{L_u}) .\]
    It is straightforward that, given $P$, for any $u \in \cM^n$, we have that $\cX^n_{L_u}$ is conditionally distributed as $Z^{L_u}$, and that these are conditionally independent of each other.
    By construction, $\cX^n$ is $\preceq$-decreasing, so that we may define
    \[ \cX := \lim_{n \to \infty} \cX^n \in \{0,1,*\}^{\Z^d} ,\]
    which exists almost surely.
    By the third property in the previous lemma, we have that $\cX \in \{0,1\}^{\Z^d}$ almost surely.
    By the first property in the lemma, it follows that $\cX$ has the same law as $\bar X$.
    Since $\cX_0 \in \{0,1\}$ almost surely, and since $\cX^n_0$ is $\preceq$-decreasing, it follows that $\cX_0=\cX^n_0$ for any $n$ such that $\cX^n_0 \neq *$. Since each $\cX^n$ is a finitary factor of $(P,U)$, it follows that so is $\cX$. This completes the proof of the theorem.
\end{proof}

\subsection{The connected case (\cref{thm:ffiid-connected})}\label{sec:connected}

Here we prove \cref{thm:ffiid-connected}, modulo the additional part about $\Z$, which will be addressed in \cref{sec:ffiid-one-dim}. Suppose that $V$ is the vertex set of an infinite connected transitive graph $G$ of degree $\Delta \ge 2$. Let $\Gamma$ be a closed transitive group of automorphisms of $G$.
We further assume that $\Gamma$ acts freely on $V$. We later explain how to remove this assumption.

Let ${\sf c}(A)$ be a connected set of minimal size containing $A$, chosen in such a way that it depends on $A$ only up to its translations by $\Gamma$, i.e., so that ${\sf c}(\gamma(A)) = \gamma({\sf c}(A))$ for $\gamma \in \Gamma$. This can be done as follows: for each equivalence class $[A]=\{ \gamma(A) : \gamma \in \Gamma \}$, choose a representative $A$, choose a connected set ${\sf c}(A)$ of minimal size containing it, and then for $A' \in [A]$, $A' \neq A$, choose $\gamma \in \Gamma$ such that $\gamma(A)=A'$ and set ${\sf c}(A') := \gamma({\sf c}(A))$. This is well defined since $\Gamma$ acts freely (this is the only place we use this assumption).

Let $N$ be a large integer to be specified later.
Let $\bar X = \bar X^0 \cup \bar X^1$, with $\bar X^0$ corresponding to sets $A$ of diameter less than $N$, and $\bar X^1$ corresponding to the remaining sets.
Define a $\{0,1\}$-valued process $\tilde X$ by
\[ \tilde X := \bigcup \{ {\sf c}(A) : X_A = 1,~\diam A \ge N\} .\]
Clearly, $\bar X^1 \subset \tilde X$.
The essential property of the connected ${\sf c}(A)$ is that the process $\tilde X$ it gives rise to  enjoys a certain spatial Markov property. For $U \subset V$, let $\partial U$ denote the external vertex boundary of $U$, i.e., vertices not in $U$ that are adjacent to $U$.
\begin{claim}
$\tilde X$ is decoupled by zeros in the sense that, for any finite $U \subset V$,
\[ \tilde X_U\text{ is conditionally independent of }\tilde X_{V \setminus U}\text{ given that }\tilde X_{\partial U} \equiv 0 .\]
\end{claim}
\begin{proof}
    Observe first that, for any set $B$, we have that $\tilde X_B \equiv 0$ is simply the event that $X_A = 0$ for all $A$ such that ${\sf c}(A) \cap B \neq \emptyset$. Thus, the conditional distribution of $(X_A)_{A \Subset V}$ given that $\tilde X_B \equiv 0$ is simply that of independent Bernoulli variables, with $X_A$ being of parameter $p_A \1_{\{{\sf c}(A) \cap B = \emptyset\}}$.
    Now take $B=\partial U$ and note that if ${\sf c}(A)$ is disjoint from $\partial U$, then $A$ is either contained in $U$ or disjoint from $U \cup \partial U$. Thus, the conditional distribution of $(\tilde X_U,\tilde X_{V \setminus U})$ given that $\tilde X_{\partial U} \equiv 0$ is the same as the unconditional distribution of $(\bigcup \{ {\sf c}(A) : X_A=1,~ A \subset U \},\bigcup \{ {\sf c}(A) : X_A=1,~ A \cap (U \cup \partial U) = \emptyset \})$. The latter are clearly independent.
\end{proof}

Let $\xi$ be a Bernoulli IID process of density $\eps$ to be specified later, independent of $X$, and define
\[ Y := \tilde X \cup \xi .\]

\begin{claim}
    $Y$ is decoupled by zeros.
\end{claim}
\begin{proof}
It is not hard to check that the union of two independent processes, each of which is decoupled by zeros, is decoupled by zeros (see \cite[Lemma 4.11]{ray2023characterizations}).
\end{proof}

\begin{claim}
    For a suitable choice of $N$ and $\eps$, we have that $Y$ is stochastically dominated by a Bernoulli IID process of density $\delta < \frac1{3\Delta-1}$, and moreover, $\P(Y_v = 1 \mid (Y_u)_{u \neq v}) \le \delta$ almost surely.
\end{claim}
\begin{proof}
    It follows from~\eqref{eq:exp-moment-cA} that for some $\lambda>\log(3\Delta-1)$,
\begin{equation}\label{eq:exp-moment-cA2}
  \sum_{A : v \in {\sf c}(A)} p_A e^{\lambda |{\sf c}(A)|} = \sum_{A : v \in A} \frac{|{\sf c}(A)|}{|A|} p_A e^{\lambda |{\sf c}(A)|} \le \sum_{A : v \in A} p_A e^{\lambda |{\sf c}(A)| + \log |{\sf c}(A)|} < \infty .
\end{equation}
Fix such a $\lambda$ and set $\eps := e^{-\lambda} < (3\Delta-1)^{-1}$. Choose $N$ large enough so that
\[ \sum_{A : v \in {\sf c}(A),\,\diam A \ge N} p_A e^{\lambda |{\sf c}(A)|} \le \frac{1}{\eps(3\Delta-1)} - 1 .\]
Using \cref{lem:domination3} (applied with $\eps_v := \eps$ to the auxiliary collection of independent Bernoulli random variables $(\tilde X_B)$ given by $\tilde X_B := \max\{X_A:{\sf c}(A)=B\}$ and $\tilde p_B := \E \tilde X_B \le \sum_{A:{\sf c}(A)=B} p_A$), we see that $Y$ is stochastically dominated by a Bernoulli IID process $Z$ of density $\delta$ given by
\[ \delta := \eps + (1-\eps)\sum_{A: v \in {\sf c}(A),\,\diam A \ge N} p_A \eps^{1-|{\sf c}(A)|} < \frac1{3\Delta-1} ,\]
and moreover, $\P(Y_v = 1 \mid (Y_u)_{u \neq v}) \le \delta$ almost surely.
\end{proof}

\begin{proof}[Proof of \cref{thm:ffiid-connected} (without the moreover part)]
Suppose that $\Gamma$ acts freely on $V$.
Choose $N$ and $\eps$ as in the previous claim.
Note that $\bar X^0$ is a block factor of IID.
It thus suffices to show that $\bar X^1$ is FFIID.
By the last two claims, $Y$ is decoupled by zeros and satisfies that $\P(Y_v = 1 \mid (Y_u)_{u \neq v}) \le \delta$ almost surely for some $\delta < \frac1{3\Delta-1}$. This puts us precisely in the setting of \cite[Theorem 4.1]{ray2023characterizations}, which says that such a process $Y$ is FFIID with a coding radius having exponential tails.

Using that $\bar X^1 \subset \tilde X \subset Y$ and suitable decoupling properties, it is now a simple matter to finish and we only give a sketch of the proof. Observe that $(\tilde X,Y)$ is decoupled by zeros of $Y$ in the sense that $(\tilde X,Y)_U$ is conditionally independent of $(\tilde X,Y)_{V \setminus U}$ given that $Y_{\partial U} \equiv 0$. In particular, given $Y$, the conditional distribution of $\tilde X$ on each cluster of $Y$ is independent and determined by the shape of the cluster (note that all clusters are finite almost surely). Thus, we may obtain $\tilde X$ as a finitary factor of $Y$ and an additional independent IID process.
Finally, $(\bar X^1,\tilde X)$ is decoupled by zeros of $\tilde X$, so that we may obtain $\bar X^1$ as a finitary factor of $\tilde X$ and an additional independent IID process.
Altogether, this shows that $\bar X^1$ is a finitary factor of an IID process. Finally, since the finitary factor to $Y$ has a coding radius with exponential tail, and since the radii of the clusters of $Y$ (and hence also of $\tilde X$) have exponential tail, the finally constructed finitary factor to $\bar X^1$ has a coding radius with exponential tail.
This completes the proof of the theorem under the assumption that $\Gamma$ acts freely on $V$.

It remains to handle the case when $\Gamma$ does not act freely on $V$.
The only use we made of the free action assumption was in order to define ${\sf c}(A)$ for all $A \Subset \Z^d$ in such a manner that our construction will be equivariant.
For this purpose, it suffices to use another independent IID process $W$ of uniform random variables to break the symmetries.
Indeed, almost surely, $\{W_v\}_{v \in V}$ are all distinct, and so are $\{\sum_{v \in U} W_v\}_{U \Subset V}$.
Using this, we may define ${\sf c}(A)$ as a finitary factor $W$ in any number of ways. To be concrete, let us describe one way to do this. Let $\cC(A)$ be the set of all connected sets containing $A$ of minimal size. Note that $|\cC(A)|<\infty$, and take ${\sf c}(A)$ to be the element $A'$ of $\cC(A)$ that minimizes $\sum_{v \in A'} W_v$.
With this definition at hand, everything else goes through just the same (the previous claims all hold conditionally on $W$).
This completes the proof of the theorem when $\Gamma$ does not act freely.
\end{proof}

\subsection{One dimension}\label{sec:ffiid-one-dim}

We define a process $\check X$ which, loosely speaking, records the chosen sets (in $X$) that cross over a given point.
Precisely, let $\cC_0$ be the collection of finite subsets of $\Z$ whose minimum is non-positive and whose maximum is non-negative, let $\cP(\cC_0)$ be the collection of subsets of $\cC_0$, and define a $\cP(\cC_0)$-valued process $\check X$ by
\[ \check X_i := \{ A \in \cC_0 : X_{A+i}=1 \} .\]
Observe that $\bar X$ is a simple block factor of $\check X$, namely, $\bar X_i$ is 1 if and only if $\check X_i$ contains a set that contains zero (compare this to $\hat X$ from~\eqref{eq:hat-X}, which is also a simple block factor of $\check X$).

\begin{thm}\label{thm:markov-chain}
    Suppose that there exists $\lambda>0$ such that
    \[ \sum_{A \Subset \Z : 0 \in A} p_A e^{\lambda \diam A} < \infty .\]
    Then $\check X$ is a countable-state mixing Markov chain with exponential return times.
\end{thm}

We break down the proof of the theorem into two steps, first showing that $\check X$ is a countable-state mixing Markov chain, and then showing that it has exponential return times. \cref{thm:markov-chain} follows immediately from this.

\begin{lemma}
    Suppose that $\sum_{A : \min A = 0} p_A \diam A < \infty$. Then $\check X$ is a countable-state mixing Markov chain.
\end{lemma}
\begin{proof}
The assumption of the lemma is necessary and sufficient in order for $\check X$ to take values in the countable state space $\cP_{\text{fin}}(\cC_0)$ of finite subsets of $\cC_0$. Indeed, $\sum_{A : \min A=0} p_A (1+\diam A)=\sum_{A \in \cC_0} p_A$ is the expected number of $A \in \cC_0$ such that $X_A=1$, and so Borel--Cantelli implies that this sum is finite if and only if the number of $A \in \cC_0$ such that $X_A=1$ is finite almost surely.

To see that $\check X$ is a Markov chain, we describe the transition probabilities. Let $\sf Z$ be a random collection of subsets of $\{0,1,\dots\}$ obtained by including each finite $A \subset \Z$ having $\min A = 0$ independently with probability $p_A$. Then it is not hard to check that
\[ \P(\check X_1 \in \cdot \mid \check X_0,~\check X_{-1},\dots) \sim \{ A-1 : A \in \check X_0,~\max A \ge 1 \} \cup {\sf Z} .\]
From this description we see that $\check X$ is a Markov chain.
Finally, it is mixing since it is a factor of an IID process in a straightforward manner.
\end{proof}

Before proving that $\check X$ has exponential return times, it is useful to observe that exponential return times do not depend on the starting state. We include a proof of this for completeness.

\begin{lemma}\label{lem:return-times-exp-tail}[exponential return time does not depend on the starting state]
    Consider a countable-state mixing Markov chain.
    Let $T_a$ be the hitting time of a state $a$ when the chain is started at $a$. If $T_a$ has exponential tail, then so does $T_b$ for every other state $b$.
\end{lemma}
\begin{proof}
    Suppose the chain $(M_n)_{n=0}^\infty$ starts at $M_0=a$.
    Set $S_0:=0$. For $i \ge 1$, let $S_i$ denote the time of the $i$-th visit to $a$ after time 0 (so $S_1=T_a$ in distribution), and $E_i$ denote the event that there is a visit to $b$ between time $S_{i-1}$ and $S_i$. Then the strong Markov property implies that $N := \min \{ i \ge 1 : E_i\text{ occurs}\}$ is a geometric random variable, whose parameter is positive (finite mean) since $M$ is mixing. Observe that the hitting time $\tau_b$ of $b$ satisfies $\tau_b \le S_N$. Thus, $\P(\tau_b \ge \lambda n) \le \P(N \ge n) + \P(S_n \ge \lambda n)$, and the latter probability decays exponentially when $\lambda>\E T_a$ since $S_n=\sum_{i=1}^n (S_i-S_{i-1})$ is the sum of IID copies of $T_a$, which has exponential tail by assumption.
    In a similar manner, $\tau'_b := \min \{ n>\tau_b : M_n=b \}$, the second hitting time of~$b$, also has exponential tail. Noting that $T_b = \tau'_b-\tau_b$ in distribution yields the lemma.
\end{proof}

\begin{lemma}
    Under the assumption of \cref{thm:markov-chain}, $\check X$ has exponential return times.
\end{lemma}
\begin{proof}
    Let $(Z_n)_{n=0}^\infty$ be a copy of the Markov chain $\check X$ started at $\emptyset$.
    Let $T$ be the first return time to $\emptyset$. We shall show that $T$ has exponential tail.

    Let $W_n := \max\{-1,\max \bigcup Z_n\}$ be the maximum of all elements in some set in $Z_n$, or $-1$ if $Z_n=\emptyset$. Note that $W_n=-1$ if and only if $Z_n=\emptyset$, so that $T$ is also the first return time to $-1$ in $W=(W_n)_{n=0}^\infty$.
    Observe also that $W$ is a Markov chain on $\{-1,0,1,\dots\}$ whose transitions obey that, given $W_{n-1}$, $W_n$ has the same law as $\max\{-1,W_{n-1}-1,\max \bigcup {\sf Z}\}$, where $\sf Z$ is as described earlier.
    Thus, we may write $W_n = \max\{W_{n-1}-1,M_n\}$ for $n \ge 1$, where $(M_n)_{n=1}^\infty$ are IID copies of $\max\{-1,\max \bigcup {\sf Z}\}$.
    Observe that $M_n$ has exponential tail because of the assumption of \cref{thm:markov-chain}.

    Fix a large positive integer $a$ and consider the stopping time $\tau_a := \min \{ n \ge 1 : W_n \le a \}$.
    Note that $W_i-W_{i-1} \le \Delta_i := \max\{-1,M_i-a\}$ whenever $W_{i-1} \ge a$.
    Thus, $W_n \le \sum_{i=1}^n \Delta_i$ whenever $\tau_a \ge n$.
    Hence, $\P(\tau_a \ge n)$ is bounded by the probability that $\sum_{i=1}^n \Delta_i \ge -1$. As the latter is a random walk with negative drift (when $a$ is large enough), whose increments are IID with exponential tail, standard concentration bounds imply that this probability decays exponentially with $n$.
    It easily follows that $T_a := \min \{ n \ge 1 : W_n = a \}$ also has exponential tail.
    The same argument shows that $T_a$ has exponential tail when the chain is started from $W_0=a$.
    Thus, the return time to $a$ has exponential tail, and it follows that $T$ has exponential tail.
\end{proof}

\begin{corollary}
    Under the assumption of \cref{thm:markov-chain}, $\check X$ is finitarily isomorphic to an IID process, and is a finitary factor of an IID process with a coding length having exponential tail.
\end{corollary}

The corollary follows immediately from \cref{thm:markov-chain} and the following known results about countable-state mixing Markov chains with exponential return times.

\begin{thm}[\cite{angel2021markov}]\label{thm:MC-ffiid}
    Let $M$ be a countable-state mixing Markov chain with exponential return times. Then $M$ is a finitary factor of an IID process with a coding length having exponential tail.
\end{thm}

\begin{thm}[\cite{rudolph1982mixing,spinka2025newprooffinitaryisomorphism}]\label{thm:MC-fin-iso-iid}
    Let $M$ be a countable-state mixing Markov chain with exponential return times. Then $M$ is finitarily isomorphic to an IID process.
\end{thm}

\cref{thm:MC-fin-iso-iid} is originally due to Rudolph~\cite{rudolph1982mixing}.
In \cite{spinka2025newprooffinitaryisomorphism}, we gave a shorter and more direct proof.

\bibliographystyle{plain}
\bibliography{library}

\end{document}